\documentclass[11pt]{amsart}
\usepackage{amsmath,mathtools,color,amssymb}
\usepackage{hyperref}

\usepackage[matrix,arrow,curve]{xy}

\newcommand{\Ga}{\mathbb{G}_a}
\newcommand{\C}{\mathbb{C}}
\newcommand{\Z}{\mathbb{Z}}

\newcommand{\Cx}{\C^\times}
\newcommand{\N}{\mathbb{N}}

\newcommand{\bR}{\mathbf{R}}

\newcommand{\cW}{{\mathcal W}}
\newcommand{\oW}{{\overline{\cW}}}
\newcommand{\g}{\mathfrak{g}}
\newcommand{\fgv}{\mathfrak g^\vee}

\newcommand{\sslash}{\mathbin{/\mkern-6mu/}}

\DeclareMathOperator{\ad}{ad}
\DeclareMathOperator{\End}{End}
\DeclareMathOperator{\gr}{gr}
\DeclareMathOperator{\Gr}{Gr}
\DeclareMathOperator{\Id}{Id}

\DeclareMathOperator{\Spec}{Spec}

\DeclareMathOperator{\wt}{wt}

\newcommand{\gkloA}{\mathbf{A}}
\newcommand{\gkloB}{\mathbf{B}}
\newcommand{\gkloC}{\mathbf{C}}
\newcommand{\cA}{\mathcal{A}}

\newcommand{\conjectural}{\widetilde{Y}}
\newtheorem{Theorem}{Theorem}[section]
\newtheorem{Corollary}[Theorem]{Corollary}
\newtheorem{Conjecture}[Theorem]{Conjecture}
\newtheorem{Proposition}[Theorem]{Proposition}
\newtheorem{Lemma}[Theorem]{Lemma}

\theoremstyle{plain}

\theoremstyle{definition}
\newtheorem{Def}[Theorem]{Definition}

\theoremstyle{remark}

\newtheorem{Rem}[Theorem]{Remark}

\title{Hamiltonian reduction for affine Grassmannian slices and truncated shifted Yangians}
\author{Joel Kamnitzer}
\address{J.~Kamnitzer: Department of Mathematics, University of Toronto, Canada}
\email{jkamnitz@math.toronto.edu}
\author{Khoa Pham}
\address{K.~Pham: Department of Mathematics, University of Toronto, Canada}
\email{khoatd.pham@mail.utoronto.ca}
\author{Alex Weekes}
\address{A.~Weekes: Department of Mathematics and Statistics, University of Saskatchewan, Canada}
\email{weekes@math.usask.ca}

\begin{document}
	\maketitle
	\begin{abstract}
		Generalized affine Grassmannian slices provide geometric realizations for weight spaces of representations of semisimple Lie algebras. They are also Coulomb branches, symplectic dual to Nakajima quiver varieties.   In this paper, we prove that neighbouring generalized affine Grassmannian slices are related by Hamiltonian reduction by the action of the additive group.  We also prove a weaker version of the same result for their quantizations, algebras known as truncated shifted Yangians.
	\end{abstract}

\section{Introduction}

\subsection{Generalized affine Grassmannian slices}
In recent years, there has been great interest in (generalized) affine Grassmannian slices, since they are geometric incarnations of weight spaces for representations of semisimple algebras.  Moreover such a slice is symplectic dual to a Nakajima quiver variety, another geometric incarnation of the same weight space.

More precisely, let $ G $ be a complex semisimple group and consider the affine Grassmannian $ \Gr_G = G((t^{-1}))/G[t] $ of $ G $.  Inside the affine Grassmannian, we have the spherical Schubert cells $ \Gr^\lambda := G[t]t^\lambda $, for $ \lambda $ a dominant coweight of $ G $. We also have transverse orbits $ \cW_\mu := G_1[[t^{-1}]]t^\mu $, for $ \mu $ a dominant coweight and we can study their intersections 
$$\oW^\lambda_\mu := \cW_\mu \cap \overline{\Gr^\lambda} $$
which are known as \textbf{affine Grassmannian slices}.

In \cite{BFN}, Braverman-Finkelberg-Nakajima generalized these spaces $\oW_\mu^\lambda$ and  $\cW_\mu$ to the case where $ \mu $ is not dominant (see Section \ref{se:WmuDef} for the definitions).  They proved that when $ G $ is simply laced, then $\oW^\lambda_\mu $ is the Coulomb branch of a quiver gauge theory, whose Higgs branch is a Nakajima quiver variety.  This was extended to the non-simply-laced case by Nakajima and the third author in \cite{NW}, to describe $\oW^\lambda_\mu$ as the Coulomb branch of a quiver gauge theory with symmetrizers.

\subsection{Relations between slices}
It is natural to expect a geometric relationship between pairs of slices $ \oW^\lambda_\mu $ and $ \oW^\lambda_{\mu + \alpha_i^\vee} $, much as Nakajima defined correspondences \cite[Definition 5.6]{Nak98} relating the corresponding pairs of quiver varieties. 

When $ G = PGL_n $ and $ \mu \le d\varpi_1^\vee $ is dominant, then the slices $ \oW^{d\varpi_1^\vee}_\mu $ are Poisson isomorphic to Slodowy slices in the nilpotent cone of $ \mathfrak{gl}_d $, by the Mirkovi\'{c}-Vybornov isomorphism \cite{MVy, MVy2, QMV}.  By the work of Gan-Ginzburg \cite{GG}, a Slodowy slice is the Hamiltonian reduction of the nilpotent cone by the action of a unipotent group.  In their PhD theses, Morgan \cite{Mor} and Rowe \cite{Ro} used these ideas to show that when $ \mu, \mu + \alpha_i^\vee $ are both dominant, then $ \oW^{d\varpi_1^\vee}_{\mu + \alpha_i^\vee} $ is the Hamiltonian reduction of $ \oW^{d\varpi_1^\vee}_\mu $ with respect to a natural $ \Ga $-action.  

In this paper, we establish such a relation for any semisimple $ G $ and any $ \lambda, \mu $.  For any simple coroot $ \alpha_i^\vee$, we define an action of $ \Ga $ on $\cW_\mu $ by left multiplication using the corresponding root subgroup which restricts to an action on $ \oW^\lambda_\mu $.  Our first main result is the following.

\begin{Theorem} \label{th:intro1}
	Let $ \lambda $ be a dominant coweight, $ \mu $ a coweight, and $ \alpha_i^\vee $ any simple coroot. There is a Poisson isomorphism $ \oW^\lambda_\mu \sslash_1 \Ga \cong \oW^\lambda_{\mu + \alpha_i^\vee} $.
\end{Theorem}

In \cite[Section 2(vi)]{BFN}, Braverman-Finkelberg-Nakajima introduced multiplication maps (see section \ref{multsection}), $ \oW^{\lambda_1}_{\mu_1} \times \oW^{\lambda_2}_{\mu_2} \rightarrow \oW^{\lambda_1 + \lambda_2}_{\mu_1 + \mu_2} $.  These are restrictions of multiplication maps $ \cW_{\mu_1} \times \cW_{\mu_2} \rightarrow \cW_{\mu_1 + \mu_2}$ \cite[section 5.9]{FKPRW}.

In order to prove the Theorem \ref{th:intro1}, we prove the following stronger result using these multiplication maps.
\begin{Theorem} \label{th:intro2}
	The multiplication map gives a $ \Ga$-equivariant Poisson isomorphism 
	$$
	\oW^0_{-\alpha_i^\vee} \times \cW_{\mu + \alpha_i^\vee} \rightarrow \Phi_i^{-1}(\Cx)_\mu \subset \cW_\mu	
	$$
	which restricts to an isomorphism
	$$
	\oW^0_{-\alpha_i^\vee} \times \oW^\lambda_{\mu + \alpha_i^\vee} \rightarrow \Phi_i^{-1}(\Cx)_\mu^\lambda \subset \oW^\lambda_\mu 
	$$
	where $ \Phi_i $ is the moment map for the $ \Ga $ action on $ \cW_\mu $.
\end{Theorem}
As $  \oW^0_{-\alpha_i^\vee} \cong T^* \C^\times $ (see Lemma \ref{le:W0alphai}), this immediately implies Theorem \ref{th:intro1}.

\subsection{Shifted Yangians}
In \cite{KWWY} and \cite[Appendix B]{BFN}, we introduced a family of algebras, called shifted Yangians $ Y_\mu $, which are used to quantize the spaces $ \cW_\mu$.  We also studied quotients, called truncated shifted Yangians $ Y^\lambda_\mu $, which quantize  the generalized affine Grassmannian slices $ \oW^\lambda_\mu $ (see section \ref{section: tsy} for the definition of these algebras).

Moreover in \cite{FKPRW}, we introduced comultiplication map $ Y_{\mu_1 + \mu_2} \rightarrow Y_{\mu_1} \otimes Y_{\mu_2} $ which were conjectured to quantize the above multiplication maps.  Here we resolve this conjecture from \cite{FKPRW}.

\begin{Theorem} \label{th:intro3}
	Upon applying associated graded, the comultiplication map $ Y_{\mu_1 + \mu_2} \rightarrow Y_{\mu_1} \otimes Y_{\mu_2} $ becomes the multiplication map $ \C[\cW_{\mu_1 + \mu_2}] \rightarrow \C[\cW_{\mu_1}] \otimes \C[\cW_{\mu_2}] $.
\end{Theorem}

With this result in hand, the following non-commutative generalization of the first part of Theorem \ref{th:intro2} is very natural and is our second main theorem.

\begin{Theorem} \label{th:intro4}
	Comultiplication gives an isomorphism
	$$
	Y_\mu[(E_i^{(1)})^{-1}] \cong Y_{-\alpha_i^\vee}^0 \otimes Y_{\mu + \alpha_i^\vee}
	$$
	and we obtain $ Y_{\mu + \alpha_i^\vee} $ as the quantum Hamiltonian reduction of $Y_\mu$.
\end{Theorem}
Here $ E_i^{(1)} $ is inverted because it is the quantum moment map for the $ \Ga $-action. The algebra $Y_{-\alpha_i^\vee}^0$ is isomorphic to the ring of differential operators on $\C^\times$ by Remark \ref{rem: cotangent bundle of C*}, a quantum version of the isomorphism $  \oW^0_{-\alpha_i^\vee} \cong T^* \C^\times $.

\subsection{Truncated shifted Yangians}
It is natural to expect a result similar to Theorem \ref{th:intro4} for the truncated shifted Yangians.  However, unlike in the commutative case, passing from $ Y_\mu $ to $ Y^\lambda_\mu $  is not immediate, since the kernel of $Y_\mu \rightarrow Y^\lambda_\mu $ is not completely understood.

In \cite{KWWY}, we gave conjectural generators of this kernel (when $ \mu $ is dominant).  In fact, we proved they generate subject to a certain ``reducedness'' conjecture for a modular version of $ \overline{\Gr^\lambda}$.  In \cite{KMWY}, for $ G = SL_n $, we proved this reducedness conjecture, and thus the conjecture about the kernel. In Appendix \ref{section: appendix}, we extend this proof to non-dominant $ \mu $ (still for $ G = SL_n$).
Applying this description of the kernel we obtain the following result for $ G = SL_n $, unconditionally, and for general $ G $, conditional on the reducedness conjecture.

\begin{Theorem} \label{th:intro5}
	Comultiplication gives an isomorphism
	$$
	Y_\mu^\lambda[(E_i^{(1)})^{-1}] \cong Y_{-\alpha_i}^0 \otimes Y_{\mu + \alpha_i^\vee}^\lambda
	$$
	and we obtain $ Y_{\mu + \alpha_i^\vee}^\lambda $ as the quantum Hamiltonian reduction of $Y_\mu^\lambda$.
\end{Theorem}

\subsection{Categorical $\fgv$-action}
Let $ \fgv $ be the Lie algebra of the group Langlands dual to $ G $. Since $ \oW^\lambda_\mu $ is a geometric incarnation of the weight space, $ V(\lambda)_\mu $, of an irreducible representation of $ \fgv $, it is natural to try to use generalized affine Grassmannian slices and their quantizations to construct categorical $ \fgv$-actions.  

To a certain extent, this was carried out in \cite{KLRYangians}, where we proved (in simply-laced type) an equivalence of categories between the category $\mathcal O $ for $ Y^\lambda_\mu $ and a category of modules for a Khovanov-Lauda-Rouquier-Webster algebra.  By \textit{transport de structure}, this leads to a categorical $ \fgv$-action on these category $ \mathcal O$s.  

This leads to the natural question of how to express this categorical action in a manner more intrinsic to the algebras $Y^\lambda_\mu$.  In order to answer this question, we would like to relate the algebras $ Y^\lambda_\mu $ and $Y^\lambda_{\mu + \alpha^\vee_i}$. This was our main motivation for developing Theorem \ref{th:intro5}.  However, we have not been able (thus far) to use Theorem \ref{th:intro5} to express the categorical $ \fgv$-action.

On the other hand, in a forthcoming work \cite{KWWYnew}, joint with Webster and Yacobi, we will construct a different relationship between $ Y^\lambda_\mu $ and $Y^\lambda_{\mu + \alpha^\vee_i}$, which is grounded in their Coulomb branch realizations.  Using this relationship, we can express the categorical $ \fgv $-action.  The exact relation between these two papers remains mysterious.

\subsection{Possible generalization}
We will now discuss Theorem \ref{th:intro2} from the Coulomb branch perspective.  Fix a reductive group $ H $ and a representation $ V $ of $ H $.  Given such a pair, Braverman-Finkelberg-Nakajima \cite{BFN1} defined a certain Poisson variety $ M_C(H,V)$, called the Coulomb branch.  

Now fix also a coweight $\gamma : \C^\times \rightarrow H $ and let $L_\gamma$ be the centralizer of $\gamma$, this is a Levi subgroup. 
Let $V^\gamma$ be the invariants for this $\C^\times$.

We learned of the following conjecture from Justin Hilburn.
\begin{Conjecture} \label{co:justin}
There is an isomorphism between $ M_C(L_\gamma,V^\gamma) $ and an open subset of $ M_C(H, V)$.
\end{Conjecture}

Fix $ G, \lambda, \mu $ as before, but assume that $ G $ is simply-laced.  Let $ H = \prod_j GL_{v_j} $ be a product of general linear groups, one for each vertex of the Dynkin diagram, with dimensions $ v_j $ determined by $ \lambda - \mu = \sum v_j \alpha^\vee_j $.  Let $ V $ be the representation space of the framed Dynkin quiver, with these gauge vertices and with framing vertices of dimensions $ w_j $, where $ \lambda = \sum_j w_j \varpi_j^\vee$.

Then by \cite{BFN}, we know that $ M_C(H, V) \cong \oW^\lambda_\mu $.  Moreover, if we choose $ \gamma = \varpi_{i,1}^\vee $ (the first fundamental weight of $ GL_{v_i}$), then we can see that $L_\gamma = \C^\times \times \prod_j GL_{v'_j} $, where $ v'_j = v_j $ if $ j \ne i $ and $ v'_i = v_i - 1$.  Moreover $ V^\gamma $ is the representation space of the framed quiver with dimension vectors $ v', w $.  Thus, we see that $ M_C(L_\gamma, V^\gamma) \cong T^* \C^\times \times \oW^\lambda_{\mu + \alpha_i^\vee} $.  Thus, we see that Theorem \ref{th:intro2} establishes Conjecture \ref{co:justin} for this case. 
More generally, Proposition 5.7 from Krylov-Perunov \cite{KrylovP} establishes Conjecture \ref{co:justin} for any choice of $ \gamma $ in this ADE quiver setting (in fact, a more general version of the conjecture where we consider $ \gamma : \C^\times \rightarrow \tilde H $, where $ \tilde H $ is an extension of $H $ by a flavour torus).

Conjecture \ref{co:justin} should also hold when $ V = 0 $.  In this case $ M_C(H,0) $ is the universal centralizer space, see \cite[Section 3(x)(a)]{BFN1}.  In \cite{FKPRW}, we constructed a map $ M_C(L_\gamma,0) \rightarrow M_C(H,0)$, which should be an isomorphism with an open subset.  Note that this map was not defined in ``Coulomb branch'' terms and thus it is not easy to see how it generalizes to the case $ V \ne 0 $.  On the other hand, a recent paper by Kato \cite{Kato} does give a ``Coulomb branch'' construction of a map $ M_C(L_\gamma,0) \rightarrow M_C(H,0)$, but for \emph{$K$-theoretic} Coulomb branches.

Conjecture \ref{co:justin} also holds when $ \gamma $ is central in $ G $; this will be a special case of the results of \cite{KWWYnew}.

\subsection{Acknowledgements}
We thank Alexander Braverman, Tudor Dimofte, Michael Finkelberg, Davide Gaiotto, Justin Hilburn, Aleksei Ilin, Dinakar Muthiah, Leonid Rybnikov, Ben Webster, and Oded Yacobi for helpful discussions.  We also thank our anonymous referees. This research was supported in part by Perimeter Institute for Theoretical Physics. Research at Perimeter Institute is supported by the Government of Canada through the Department of Innovation, Science and Economic Development Canada and by the Province of Ontario through the Ministry of Research, Innovation and Science.

\section{Generalized affine Grassmannian slices}\label{reltoWmu}

\subsection{Notation}
\label{section:notation}
Let $G$ be a connected complex semisimple group  with Lie algebra $\mathfrak{g}$. Let $T$ be a maximal torus of $G$, $B$ a Borel subgroup, and $B_-$ the opposite Borel subgroup. Consider $U$ (resp. $U_-$) the unipotent radical of $B$ (resp. $B_-$).  Denote by $\g,\mathfrak{h},  \mathfrak{b}^\pm, \mathfrak{u}^\pm$ the corresponding Lie algebras.   Let $W$ be the Weyl group, and $\Delta$ the set of roots.  Write $\{\alpha_i\}_{i\in I}$ for the set of simple roots, $\{\alpha_i^\vee\}_{i\in I}$ the simple coroots, and $\{\varpi_i\}_{i \in I}$ and $\{\varpi_i^\vee\}_{i\in I}$ the fundamental (co)weights.   Write $\langle \cdot, \cdot\rangle: \mathfrak{h}\times \mathfrak{h}^\ast \rightarrow \C$ for the natural pairing.

Let $(a_{ij})_{i,j \in I}$ be the Cartan matrix, and let $d_i$ be the unique coprime positive integers such that $d_i a_{ij}$ is symmetric.  Define a bilinear form on $\mathfrak{h}^\ast$ by $(\alpha_i, \alpha_j) =  d_i a_{ij}$.  We will also denote this form by $\alpha_i \cdot \alpha_j$.

Let $e_i, f_i, h_i$ be the Chevalley generators for $\g$.  There is a non-degenerate symmetric invariant  bilinear form $(\cdot,\cdot)$ on $\g$ defined by $(e_i, f_j) = \delta_{ij} d_i^{-1}$ and $(h_i, h_j) = d_j^{-1}a_{ij}$.

For each $ i$ we have a homomorphism $ \tau_i : SL_2 \rightarrow G $.  In particular we can restrict this map to the upper (resp. lower) triangular subgroups of $ SL_2 $ and we denote these restrictions $ x_i : \Ga \rightarrow G $ (resp. $ x_{-i} : \Ga \rightarrow G $); this exponentiates the element $ e_i $ (resp. $f_i$).

For each simple root $\alpha_i $, we have a group homomorphism $ \psi_i : U \rightarrow \Ga$ (and also $ \psi_i^- : U_- \rightarrow \Ga$).  We normalize $\psi_i$ so that the composition of $ x_i $ with $ \psi $ is the identity (and similarly for $ \psi_i^-$).

We extend $\psi_i $ to $ UTU_- \subset G $ by defining $ \psi_i(uhv) = \psi_i(u) $.

Let $X$ be an affine Poisson scheme over $\C$, and suppose that an algebraic group $H$ acts on $X$ by Poisson automorphisms.  Recall that this action is called Hamiltonian if there exists a moment map $\mu: \operatorname{Lie}(H) \rightarrow \C[X]$, which is a homomorphism of Lie algebras, such that
$$
\{\mu(h), f \}(x) = \frac{\partial}{\partial \epsilon} f \big( \exp( - \epsilon h) x\big) \big|_{\epsilon = 0}
$$
for all $h\in \operatorname{Lie}(H), f\in \C[X]$ and $x\in X$.  If $\xi \in \operatorname{Lie}(H)^\ast$ is $H$-invariant, we denote the corresponding Hamiltonian reduction by
$$
X \sslash_\xi H = \operatorname{Spec} \left( \raisebox{.2em}{$\C[X]$}\left/\raisebox{-.2em}{$\langle \mu(h)- \xi(h) : h \in \operatorname{Lie} H \rangle$}\right. \right)^H 
$$

\subsection{Loop and arc spaces}
\label{section: loops and arcs}
	We shall briefly discuss the topic of arc and loop spaces, following \cite{Fr,Gortz}.  For our purposes it will be sufficient to consider loops into affine schemes.  Let $ A $ be a finitely-generated commutative $\C$-algebra, 
	$$A=\mathbb{C}[x_1,\dots ,x_n]/\langle f_1, \dots, f_m\rangle$$ 
	where the $f_i$'s are polynomials in $x_j$'s. Let $X=\Spec A$ be the corresponding finite-type affine scheme. 
	
	Let $d\geq 0$. The arc spaces $X[t]$, $X[t]_{\leq d}$ and $X[[t]]$ are defined in terms of their functors of points: for any $\mathbb{C}$-algebra $R$,
	\begin{align*}
	X[t](R)&=X(R[t])=\hom_\mathbb{C}(A,R[t])= \big\{ (x_1(t),\dots ,x_n(t))\in R[t]^n: f_i(x_j(t))=0\big\},\\
	X[t]_{\leq d}(R)&= \big\{ (x_1(t),\dots ,x_n(t))\in R[t]^n: f_i(x_j(t))=0, \deg(x_j(t))\leq d\big\}\subseteq X[t](R),\\
	X[[t]](R)&=X(R[[t]])=\hom_\mathbb{C}(A,R[[t]])= \big\{ (x_1(t),\dots ,x_n(t))\in R[[t]]^n: f_i(x_j(t))=0\big\}.
	\end{align*}
We may also define the loop space $X((t))$ by the functor
\begin{align*}
X((t))(R) &= X\big( R((t)) \big) = \hom_\mathbb{C}\big(A, R((t)) \big)\\
& =\big\{ (x_1(t),\dots ,x_n(t))\in R((t))^n: f_i(x_j(t))=0\big\}.
\end{align*}
Then $X[[t]]$ and $X[t]_{\leq d}$  are affine schemes (see Lemma \ref{lemma: arc spaces as affine schemes}), while $X[t]$ and $X((t))$ are ind-schemes. 

\begin{Rem}
We will also often work with the scheme $X[[t^{-1}]]$ and the ind-schemes $X[t^{-1}]$ and $X((t^{-1}))$. They are defined analogously to their counterparts $X[[t]], X[t]$ and $X((t))$ above, and enjoy analogous properties.
\end{Rem}

It will be useful to work using explicit coordinates.	We write $x_j(t)=\sum_r x_j^{(r)}t^r$. Consider the polynomial ring $\mathbb{C}[x_j^{(r)}: 1\leq j\leq n, r> 0]$. Following \cite[3.4.2]{Fr}, define a derivation $T$ on this polynomial ring by $T(x_j^{(r)})=(r+1)x_j^{(r+1)}$. For $1\leq i\leq m$, let $\tilde{f}_i\in \mathbb{C}[x_j^{(r)}: 1\leq j\leq n, r\ge0]$ be the same polynomials as $f_i$, with $x_j$ replaced by $x_j^{(0)}$. Next, set
	\[
	B= \mathbb{C}[x_j^{(r)}: 1\leq n\leq j, r\geq 0]/\langle T^k\tilde{f_i} : 1\leq i\leq m, k\geq 0\rangle.
	\]
	and let $I_d=\langle x_j^{(r)}: r> d\rangle \subseteq B$.
The following result is \cite[3.4.2]{Fr}.

\begin{Lemma}
\label{lemma: arc spaces as affine schemes}
	The functors $X[[t]], X[t]_{\leq d} $ are represented by the schemes $  \Spec B$ and $ \Spec B/I_d $.
\end{Lemma}

We introduce some notation before our next result.  Let $Y = \Spec A$ be an affine scheme, and suppose that $\{Y_d\}$ is a collection of closed subschemes of $Y$, where $d$ runs over some index set.  Then we can write each $Y_d = \Spec A / J_d$ for some ideal $J_d \subset A$.  We will say that the collection $\{Y_d\}$ \textbf{fills-out} $Y$ if the intersection $\bigcap_d J_d = 0$.  In other words, this means that the ideals $J_d$ separate functions: two elements $f,g \in A$ are equal iff $f-g \in J_d$ for all $d$. Alternatively, it means that there is no proper closed subscheme $Z\subset Y$ which contains all $Y_d$ as closed subschemes.

\begin{Lemma}
\label{lemma: fills-out}
The collection of subschemes $\{ X[t]_{\leq d} \}_{d \geq 0}$ fills-out $X[[t]]$.
\end{Lemma}
\begin{proof}
In view of the previous lemma, we need to show that $\bigcap_d I_d=\{0\}$ inside the ring $B$. Observe that $B$ is a graded ring by setting $\deg (x_j^{(r)})=r$, and that $I_d$ is homogeneous and lies in degree greater than $d$. Therefore, $\bigcap_d I_d$ is also homegeneous and lies in degree greater than $d$ for all $d$. Therefore, $\bigcap_d I_d=\{0\}$.
\end{proof}

The proof of the previous Lemma immediately generalizes to the following result.
\begin{Lemma}\label{lemma: fills-out2}
	Suppose that $A = \bigoplus_{n \in \Z} A_n$ is a graded ring, and the ideals $J_d$ are homogeneous.  Let $ Y = \Spec A $ and let $ Y_d = \Spec A/J_d $.  Assume also that the integers $n_d := \min \{ \deg a : a \in J_d\}$ are well-defined, and that $\{n_d\}$ is unbounded above.  Then $\{Y_d\}$ fills-out $Y$.
\end{Lemma}

\begin{Rem}
The notion of filling-out is closely related to Zariski density, but not equivalent. For example, take $Y = \Spec \C[x] $ and $Y_d = \Spec \C[x] / \langle x^d \rangle$. Then the collection $\{Y_d\}$ fills-out $Y$, but the union $\bigcup_d Y_d$ is not Zariski dense.
\end{Rem}

%%%%%%
\subsection{The scheme $\mathcal{W}_\mu$}\label{Wmudef}
\label{se:WmuDef}
For any affine algebraic group $H$, we may form the loop space $H[[t^{-1}]]$ as in the previous section, which is an infinite-dimensional affine scheme.  We denote by $H_1[[t^{-1}]]$ the kernel of the evaluation $H[[t^{-1}]] \rightarrow H$ at $t^{-1}=0$.  Then $H_1[[t^{-1}]] \subset H[[t^{-1}]]$ is a closed subscheme.  

Let $\mu: \mathbb{G}_m \rightarrow T$ be any coweight. Denote by $t^\mu$ its image in $G((t^{-1}))$. We are interested in the following subscheme of $ G((t^{-1})) $.
\begin{equation}
\label{eq: def of Wmu}
\mathcal{W}_\mu:= U_1[[t^{-1}]]T_1[[t^{-1}]]t^\mu U_{-,1}[[t^{-1}]].
\end{equation}

\begin{Rem} \label{re:usual}
	When $ \mu $ is dominant, then the natural map $ G((t^{-1})) \rightarrow G((t^{-1}))/G[t]$ restricts to an isomorphism betweeen $ \mathcal W_\mu $ and the usual affine Grassmannian slice $ G_1[[t^{-1}]] t^\mu \subset G((t^{-1}))/G[t]$ (see \cite[Remark 5.10]{FKPRW}).
	\end{Rem}

The inclusion $U_1[[t^{-1}]]\rightarrow U((t^{-1}))$ gives rise to the isomorphism $U_1[[t^{-1}]]\simeq U((t^{-1}))/U[t]$, see \cite[Lemma 5.4]{Krylov}. So, there is a morphism
\[
\pi:  U((t^{-1})) T_1[[t^{-1}]] t^\mu U_-((t^{-1})) \longrightarrow \mathcal{W}_\mu .
\]
defined by $ \pi(u ht^\mu u_-) = n u ht^\mu u_- n_- $ where $ n \in U[t], n_- \in U_-[t] $ are chosen (uniquely) so that $ nu \in U_1[[t^{-1}],  u_- n_- \in U_{-,1}[[t^{-1}]] $.

We record, for posterity, the following easy lemma.

\begin{Lemma}
	Let $ g \in U((t^{-1})) T_1[[t^{-1}]] t^\mu U_-((t^{-1}))$ and $ n \in U[t], n_- \in U_-[t] $.  Then we have $\pi(ngn_-) = \pi(g) $.
\end{Lemma}

	There is a family of related spaces, called \textbf{generalized affine Grassmannian slices}:  Given a dominant coweight $\lambda$ of $ G$, consider the subscheme $\overline{\Gr^\lambda} := \overline{G[t]t^\lambda} \subset \Gr_G = G((t^{-1}))/G[t]$, the spherical Schubert variety.  	 Then we define $ \overline{G[t] t^\lambda G[t]}$ as the preimage of $\overline{\Gr^\lambda}$ under $G((t^{-1}))\rightarrow \Gr_G $.  On the level of $ \C$-points, we have that 
	$$ \overline{G[t]t^\lambda G[t]} = \bigcup_{\lambda' \le \lambda} G[t]t^{\lambda'}G[t]$$

For any dominant coweight $ \lambda $ and any coweight $\mu \le \lambda$, we define the generalized slice to be $$\overline{\mathcal{W}}^\lambda_\mu=\mathcal{W}_\mu \cap \overline{G[t]t^\lambda G[t]}$$
This intersection is taken inside the ind-scheme $G((t^{-1}))$, and in particular $\oW_\mu^\lambda $ is a closed subscheme of $\cW_\mu$.  Equivalently, $ \oW^\lambda_\mu $ is the fibre product of $ \cW_\mu $ and $ \overline{\Gr^\lambda} $ over $ \Gr_G $.

These generalized affine Grassmannian slices were introduced in \cite{BFN} where the following results were proven.
\begin{Theorem}
\begin{enumerate}
\item $ \oW^\lambda_\mu$ admits a modular description in terms of certain principal $G$-bundles on $ \mathbb P^1 $ equipped with $B$ and $ B_-$ structures, as described in \cite[section 2(ii)]{BFN}.
\item $ \oW^\lambda_\mu $ is an affine variety of dimension $ \langle \lambda - \mu, 2 \rho \rangle $, where $ 2\rho$ is the sum of the positive roots. 
\item In ADE types, $\oW^\lambda_\mu $ is a Coulomb branch associated to a quiver gauge theory, by \cite[Theorem 3.10]{BFN}.  
In BCFG types, it is a Coulomb branch associated to a quiver gauge theory with symmetrizers, by \cite[Theorem 4.1]{NW}.
\item When $ \mu $ is dominant, then $ \oW^\lambda_\mu $ is isomorphic to a usual affine Grassmannian slice, see Remark \ref{re:usual} above and \cite[Remark 2.9]{BF}.
\end{enumerate}
\end{Theorem}
The varieties $\oW_{-\alpha_i^\vee}^0$ will play an important role in the present paper, see e.g.~Proposition \ref{mequivariant} and Corollary \ref{cor: reduction for slices v1}.  It follows from Remark \ref{rem: cotangent bundle of C*} below that for any $i\in I$ there is a Poisson isomorphism
$$
\oW_{-\alpha_i^\vee}^0 \cong T^\ast \C^\times,
$$
Here the left-hand side is endowed with the Poisson structure coming from the truncated shifted Yangian $Y_{-\alpha_i^\vee}^0$ (see Section \ref{section: Poisson struture on generalized slices}), while the right-hand side is equipped with its canonical Poisson structure as a cotangent bundle. (Alternatively, we know for any $ \mu \le 0$ that $ \oW_\mu^0 $ is the moduli space of $G$-monopoles of charge $ -\mu $, whose symplectic structure was studied in \cite{FFKM}, and the above isomorphism is a special case of the results of that paper.)

If $\mu \leq \lambda \leq \nu$ where $\lambda,\nu$ are both dominant, then we have a natural inclusion $\oW_\mu^\lambda \subseteq \oW_\mu^\nu$ as subschemes of $\cW_\mu$.  

%%%%%%%%%%%%
\subsection{Filling-out $\cW_\mu$}
We have a collection of closed subschemes $\{\overline{\cW}_\mu^\lambda\}$ of $\cW_\mu$, where $\lambda$ runs over dominant coweights with $\lambda \geq \mu$. Our goal in this section is to prove the following result:

\begin{Proposition}\label{denseunion}
The collection of subschemes $\{\overline{\mathcal{W}}^\lambda_\mu\}$  fills-out $\mathcal{W}_\mu$.
\end{Proposition}

\begin{proof}

Fix a closed embedding $G \subseteq \C^n$ for some $n$, i.e.~a choice of generators of the coordinate ring of $G$.  Then $U\subset G$ inherits a closed embedding $U \subset \C^n$, as do $T$ and $U_-$. For $d \geq 0$, we define the closed subscheme
$$
U_1[t^{-1}]_{\leq d} \ :=\ U_1[[t^{-1}]] \cap U[t^{-1}]_{\leq d} \ \subset \ U[[t^{-1}]]
$$
where $U[t^{-1}]_{\leq d}$ is defined as in Section \ref{section: loops and arcs}.    Similarly we define collections  $T_1[t^{-1}]_{\leq d}$ and $U_{-,1}[t^{-1}]_{\leq d}$.  Taking products, we define a closed subscheme of $\cW_\mu$ by
$$
\cW_{\mu, \leq d} \ =\  U_1[t^{-1}]_{\leq d} T_1[t^{-1}]_{\leq d}  t^\mu U_{-,1}[t^{-1}]_{\leq d}
$$ 

The inclusion $\cW_{\mu,\leq d} \hookrightarrow \cW_\mu$ is isomorphic to the inclusion
$$
U_1[t^{-1}]_{\leq d} \times T_1[t^{-1}]_{\leq d}  \times U_{-,1}[t^{-1}]_{\leq d} \ \hookrightarrow \ U_1[[t^{-1}]] \times T_1[[t^{-1}]] \times U_{-,1} [[t^{-1}]]
$$
Note that the coordinate ring of the right hand side here carries a grading, analogous to that defined for $X[[t]]$ in the proof of Lemma \ref{lemma: fills-out}. 
Applying Lemma \ref{lemma: fills-out2}, we see that $ U_1[t^{-1}]_{\leq d} \times T_1[t^{-1}]_{\leq d}  \times U_{-,1}[t^{-1}]_{\leq d}  $ fills-out $ U_1[[t^{-1}]] \times T_1[[t^{-1}]] \times U_{-,1} [[t^{-1}]]$.  Hence the collection $\{\cW_{\mu,\leq d}\}$ fills-out $\cW_\mu$.

We claim that for each $d$, there exists a $\lambda$ such that $\cW_{\mu, \leq d} \subseteq \overline{\cW}_\mu^\lambda$. Assuming this claim for the moment, suppose that $Z \subset \cW_\mu$ is a closed subscheme which contains all $\overline{\cW}_\mu^\lambda$.  Then by the claim, $Z$ also contains all $\cW_{\mu,\leq d}$.  Since $\{\cW_{\mu,\leq d}\}$ fills-out $\cW_\mu$ this implies that $Z = \cW_\mu$, and therefore $\{\overline{\cW}_\mu^\lambda\}$ also fills out $\cW_\mu$.

To prove the claim, observe that $\cW_{\mu, \leq d} \subset \cW_\mu \cap G[t,t^{-1}]$. Therefore we get a map $\cW_{\mu,\leq d} \rightarrow G[t,t^{-1}] \rightarrow G[t,t^{-1}] / G[t] =: \Gr^{thin}_G$ into the thin affine Grassmannian.  Recall that $\Gr^{thin}_G = \lim_{\rightarrow} \overline{\Gr^\lambda}$ presents $\Gr_G^{thin}$ as an ind-scheme.  Since $\cW_{\mu, \leq d}$ is an affine scheme and in particular is quasi-compact, it follows from \cite[Lemma 2.4]{Gortz} that the map $\cW_{\mu,\leq d} \rightarrow \Gr_G$ factors through $\cW_{\mu,\leq d} \rightarrow \overline{\Gr^\lambda}$ for some $\lambda$.  By the definition of $\overline{\cW}_\mu^\lambda$ this implies that $\cW_{\mu, \leq d} \subseteq \overline{\cW}_\mu^\lambda$, as claimed.
\end{proof}

\subsection{Shift and multiplication maps between the varieties $\mathcal{W}_\mu$}\label{multsection}

In this section, we look at some natural maps between the varieties $\mathcal{W}_\mu$, namely the multiplication maps and the shift maps, following \cite[Section 5.9]{FKPRW}.

\begin{Def}
\label{def: shift geometric}
	Let $\mu_1,\mu_2$ be antidominant coweights. We define the shift map $\iota_{\mu, \mu_1,\mu_2}: \mathcal{W}_{\mu+\mu_1+\mu_2} \rightarrow \mathcal{W}_\mu$, $g\mapsto \pi(t^{-\mu_1}gt^{-\mu_2})$ where $\pi$ is as in section \ref{Wmudef}.
\end{Def}

\begin{Def}
	For any coweights $\mu_1$ and $\mu_2$, we define the multiplication map $m_{\mu_1,\mu_2}:  \mathcal{W}_{\mu_1} \times \mathcal{W}_{\mu_2}  \rightarrow \mathcal{W}_{\mu_1+\mu_2}$ as $(g_1,g_2)\mapsto \pi(g_1g_2)$.
\end{Def}

We recall the following compatiblity between the shift and multiplication maps.

\begin{Lemma}\emph{\cite[Lemma~5.11]{FKPRW}} \label{shiftmultcomp}
	Let $\mu_1,\mu_2$ be any coweights and let $\nu_1,\nu_2$ be antidominant coweights. The following diagram commutes.
	\[
	\xymatrix{
		\mathcal{W}_{\mu_1+\nu_1} \times \mathcal{W}_{\mu_2+\nu_2} \ar[rr]^{m_{\mu_1+\nu_1,\mu_2+\nu_2}} \ar[d]_{\iota_{\mu_1,\nu_1,0}\times \iota_{\mu_2,0,\nu_2}} & & \mathcal{W}_{\mu_1+\mu_2+\nu_1+\nu_2} \ar[d]^{\iota_{\mu_1+\mu_2,\nu_1,\nu_2}} \\
		\mathcal{W}_{\mu_1} \times \mathcal{W}_{\mu_2} \ar[rr]_{m_{\mu_1,\mu_2}}& & \mathcal{W}_{\mu_1+\mu_2}
	}
	\]
\end{Lemma}

We will see later that the shift and multiplication maps restrict to maps for the subvarieties $\overline{\cW}_\mu^\lambda$.  See Propositions \ref{shiftpoi} and \ref{prop: mult map restricts to slices}, respectively.

\section{Truncated shifted Yangians}
Following \cite[Appendix B]{BFN}, we introduce a family of algebras known as the truncated shifted Yangians. These algebras are used to quantize the generalized affine Grassmannian slices.
\subsection{Shifted Yangians}
Recall the bilinear form $\alpha_i \cdot \alpha_j = (\alpha_i, \alpha_j) = d_i a_{ij}$ from section \ref{section:notation}.

\begin{Def}\label{cartandoubleyangian}
The {\em Cartan doubled Yangian} $Y_\infty := Y_\infty(\mathfrak{g})$ is defined to be the $ \mathbb{C} $--algebra with generators $ E_i^{(q)}, F_i^{(q)}, H_i^{(p)} $ for $ i\in I$, $ q > 0 $ and $ p \in \mathbb{Z} $, with relations
\begin{align}
\label{H,H} [H_i^{(p)}, H_j^{(q)}] &= 0,  \\
\label{E,F} [E_i^{(p)}, F_j^{(q)}] &=  \delta_{ij} H_i^{(p+q-1)}, \\
\label{H,E} [H_i^{(p+1)},E_j^{(q)}] - [H_i^{(p)}, E_j^{(q+1)}] &= \frac{\alpha_i \cdot \alpha_j}{2} (H_i^{(p)} E_j^{(q)} + E_j^{(q)} H_i^{(p)}) , \\
\label{H,F} [H_i^{(p+1)},F_j^{(q)}] - [H_i^{(p)}, F_j^{(q+1)}] &= -\frac{\alpha_i \cdot \alpha_j}{2} (H_i^{(p)} F_j^{(q)} + F_j^{(q)} H_i^{(p)}) , \\
\label{E,E} [E_i^{(p+1)}, E_j^{(q)}] - [E_i^{(p)}, E_j^{(q+1)}] &= \frac{\alpha_i \cdot \alpha_j}{2} (E_i^{(p)} E_j^{(q)} + E_j^{(q)} E_i^{(p)}), \\
\label{F,F} [F_i^{(p+1)}, F_j^{(q)}] - [F_i^{(p)}, F_j^{(q+1)}] &= -\frac{\alpha_i \cdot \alpha_j}{2} (F_i^{(p)} F_j^{(q)} + F_j^{(p)} F_i^{(q)}),\\
i \neq j,\  N = 1 - a_{ij} \Longrightarrow
\label{symE} \operatorname{sym} &[E_i^{(p_1)}, [E_i^{(p_2)}, \cdots [E_i^{(p_N)}, E_j^{(q)}]\cdots]] = 0, \\
i \neq j,\ N = 1 - a_{ij} \Longrightarrow
\label{symF} \operatorname{sym} &[F_i^{(p_1)}, [F_i^{(p_2)}, \cdots [F_i^{(p_N)}, F_j^{(q)}]\cdots]] = 0.
\end{align}

\end{Def}

\begin{Def}\label{shiftedyangian}
For any coweight $\mu$, the shifted Yangian $Y_\mu$ is defined to be the quotient of $Y_\infty$ by the relations $H_i^{(p)}=0$ for all $p<-\langle \mu,\alpha_i\rangle$ and $H_i^{(-\langle \mu,\alpha_i\rangle)}=1$.
\end{Def}

\begin{Rem}
When $\mu=0$, the algebra $Y=Y_0$ is the usual Yangian. The above generators and above relations correspond to Drinfeld's ``new'' presentation of $Y$, see \cite[Theorem 12.1.3]{CP}.
\end{Rem}

It is convenient to define series version of these generators.
\begin{equation}
\label{eq: series forms}
H_j(t)= t^{\langle \mu,\alpha_j\rangle} + \sum_{r\geq 1} H_j^{(-\langle \mu,\alpha_j\rangle + r)}t^{\langle \mu,\alpha_j\rangle-r}, \quad E_j(t)=\sum_{r\geq 1} E_j^{(r)}t^{-r}, \quad  F_j(t)=\sum_{r\geq 1} F_j^{(r)}t^{-r}.
\end{equation}

We can relate these algebras in a natural way, via ``shift morphisms".
\begin{Proposition}\label{shiftmaps}\emph{\cite[Prop~3.8, Cor.~3.16]{FKPRW}}
Let $\mu$ be a coweight, and $\mu_1,\mu_2$ be antidominant coweights. Then there exists a homomorphism $\iota_{\mu,\mu_1,\mu_2}: Y_\mu \rightarrow Y_{\mu+\mu_1+\mu_2}$ defined by
\begin{equation}
H_i^{(r)}\mapsto H_i^{(r- \langle \mu_1+\mu_2,\alpha_i\rangle)},\ \ E_i^{(r)}\mapsto E_i^{(r-\langle \mu_1,\alpha_i\rangle)},\ \ F_i^{(r)}\mapsto F_i^{(r-\langle \mu_2 ,\alpha_i\rangle)}.
\end{equation}
Moreover, this homomorphism  $\iota_{\mu,\mu_1,\mu_2}: Y_\mu \rightarrow Y_{\mu+\mu_1+\mu_2}$ is injective.
\end{Proposition}

\begin{Rem}
In \cite{KWWY}, for $\mu$ dominant, the shifted Yangian $Y_\mu$ is realized as a subalgebra of the usual Yangian $Y_0$ and not as a quotient of $Y_\infty$. In our setting, the shift morphism $\iota_{\mu,0,-\mu}$ corresponds to the natural inclusion $Y_\mu \rightarrow Y_0$ in \cite{KWWY}.
\end{Rem}

Next, let us introduce the following elements of the shifted Yangians, similar to certain elements of the usual Yangian considered by Levendorskii in \cite{Le1}.
\begin{Def}\label{lnH}
Set $S_i^{(-\langle \mu,\alpha_i\rangle +1)} =H_i^{(-\langle \mu,\alpha_i\rangle +1)}$ and
\begin{equation}
S_i^{(-\langle \mu,\alpha_i\rangle +2)} = H_i^{(-\langle \mu,\alpha_i\rangle +2)} - \tfrac{1}{2}\big(H_i^{(-\langle \mu,\alpha_i\rangle +1)}\big)^2
\end{equation}
\end{Def}
It is not hard to check that for $ r \ge 1 $,
\begin{align*}
[S_i^{(-\langle \mu, \alpha_i \rangle +2)} ,E_j^{(r)}]&=(\alpha_i \cdot \alpha_j)  E_j^{(r+1)}, \\
[S_i^{(-\langle \mu, \alpha_i \rangle +2)} ,F_j^{(r)}]&=-(\alpha_i \cdot \alpha_j)  F_j^{(r+1)}.
\end{align*}
These commutation relations show that $ [S_i^{(-\langle \mu, \alpha_i \rangle +2)}, ?] $ plays the role of a ``raising operator", allowing us to obtain higher $E_i^{(r)}$ and $F_i^{(r)}$ from $E_i^{(1)}$ and $F_i^{(1)}$.

\begin{Lemma}\label{4generators}\emph{\cite[Lem~3.11]{FKPRW}}
Let $\mu$ be an antidominant coweight. As a unital associative algebra, $Y_\mu$ is generated by $E_i^{(1)},F_i^{(1)},S_i^{(-\langle \mu,\alpha_i\rangle + 1)}, S_i^{(-\langle \mu,\alpha_i\rangle +2)}$.
Alternatively, $Y_\mu$ is also generated by $E_i^{(1)},F_i^{(1)}, H_i^{(-\langle \mu,\alpha_i\rangle +k)}$ ($k=1,2$). In particular, $Y_\mu$ is finitely generated.
\end{Lemma}

\subsection{PBW theorem}
In this section, we describe the PBW theorem for shifted Yangians, generalizing the case of ordinary Yangian (due to Levendorskii in \cite{L2}), and the case of dominantly shifted Yangians \cite[Prop~3.11]{KWWY}.

\begin{Def}\label{egamma}
Let $\beta$ be a positive root, and pick any decomposition $\beta = \alpha_{i_1} +\ldots + \alpha_{i_l}$ into simple roots so that the element $ [e_{i_1}, [e_{i_2}, \ldots, [ e_{i_{l-1}}, e_{i_l}]\cdots] $ is a non-zero element of the root space $\mathfrak{g}_\beta$.  Consider also $q>0$ and a decomposition $q + l - 1 = q_1 +\ldots + q_l$ into positive integers.  Then we define a corresponding element of $Y_\infty$:
\begin{equation}
E_\beta^{(q)} :=  [ E_{i_1}^{(q_1)}, [E_{i_2}^{(q_2)},\ldots [E_{i_{l-1}}^{(q_{l-1})}, E_{i_l}^{(q_l)}]\cdots ].
\end{equation}
This element $E_\beta^{(q)}$, called a \emph{PBW variable}, depends on the choices above.  However, we will fix such a choice for each $ \beta $ and $ q $.

Similarly, we define and fix a PBW variable $F_\beta^{(q)}$ for each positive root $\beta$ and each $q>0$.
\end{Def}
Choose a total order on the set of PBW variables
\begin{equation}\label{total-order-generators}
\left\{ E_\beta^{(q)} : \beta\in \Delta^+, q>0\right\} \cup \left\{ F_\beta^{(q)} : \beta\in \Delta^+, q>0 \right\} \cup \left\{ H_i^{(p)} : i\in I, p> -\langle\mu, \alpha_i\rangle \right\}
\end{equation}

For simplicity we will assume that we have chosen a block order with respect to the three subsets above, i.e. ordered monomials have the form $EFH$.

\begin{Theorem}\emph{\cite[Cor.~3.15]{FKPRW}}\label{PBW}
For $\mu$ arbitrary, the set of ordered monomials in PBW variables form a PBW basis for $Y_\mu$ over $\mathbb{C}$.
\end{Theorem}

\subsection{Coproducts for shifted Yangians}
In \cite{FKPRW} (building on the work of \cite{GNW} for the case $\mu = 0$), we defined a family of coproducts for shifted Yangians. We recall these definitions here.  Although these results were proven only for simply-laced types in \cite{FKPRW}, their proofs apply in general.

\begin{Theorem}\emph{\cite[Theorem 4.8]{FKPRW}}\label{th:coproductanti}
Let $\mu, \mu_1$, and $\mu_2$ be antidominant coweights. There exists an algebra morphism $\Delta_{\mu_1,\mu_2}:Y_{\mu_1+ \mu_2} \longrightarrow Y_{\mu_1} \otimes Y_{\mu_2}$ defined by
\begin{align*}
\Delta(E_i^{(r)})&= E_i^{(r)}\otimes 1, \hskip1ex 1\leq r \leq -\langle \mu_1,\alpha_i \rangle ;\\
\Delta(E_i^{(-\langle \mu_1,\alpha_i \rangle+1)})&=E_i^{(-\langle \mu_1 ,\alpha_i \rangle +1)}\otimes 1 + 1\otimes E_i^{(1)};\\
\Delta(E_i^{(-\langle \mu_1,\alpha_i \rangle+2)})&=E_i^{(-\langle \mu_1,\alpha_i\rangle+2)}\otimes1+1\otimes E_i^{(2)}+S_i^{(-\langle \mu_1,\alpha_i\rangle+1)}\otimes E_i^{(1)}\\
&\hskip2ex -\sum\limits_{\gamma>0}F_\gamma^{(1)}\otimes[E_i^{(1)},E_\gamma^{(1)}];
\end{align*}
\begin{align*}
\Delta(F_i^{(r)})&= 1\otimes F_i^{(r)}, \hskip1ex 1\leq r \leq -\langle \mu_2,\alpha_i \rangle; \\
\Delta(F_i^{(-\langle \mu_2,\alpha_i \rangle+1)})&=1\otimes F_i^{(-\langle \mu_2 ,\alpha_i \rangle +1)}+ F_i^{(1)}\otimes 1;\\
\Delta(F_i^{(-\langle \mu_2,\alpha_i \rangle+2)})&=1\otimes F_i^{(-\langle \mu_2,\alpha_i\rangle+2)}+ F_i^{(2)}\otimes 1 +F_i^{(1)}\otimes S_i^{(-\langle \mu_2,\alpha_i\rangle+1)}\\
&\hskip2ex +\sum\limits_{\gamma>0}[F_i^{(1)},F_\gamma^{(1)}]\otimes E_\gamma^{(1)};\\
\Delta(S_i^{(-\langle \mu,\alpha_i\rangle +1 )})&=S_i^{(-\langle \mu_1,\alpha_i\rangle +1 )}\otimes 1 + 1\otimes S_i^{(-\langle \mu_2,\alpha_i\rangle +1 )};\\
\Delta(S_i^{(-\langle \mu,\alpha_i\rangle +2 )})&= S_i^{(-\langle \mu_1,\alpha_i\rangle +2 )}\otimes 1 + 1\otimes S_i^{(-\langle \mu_2,\alpha_i\rangle +2 )} - \sum\limits_{\gamma>0}( \alpha_i \cdot\gamma )F_\gamma^{(1)} \otimes E_\gamma^{(1)}.
\end{align*}
\end{Theorem}

If we have arbitrary (not necessarily antidominant) coweights, we also have a coproduct, whose behaviour is determined by the antidominant case.

\begin{Theorem}\emph{\cite[Theorem 4.12]{FKPRW}} \label{coproduct}
Let $\mu=\mu_1+\mu_2$ where $\mu,\mu_1,\mu_2$ are arbitrary coweights. There exists a unique coproduct $\Delta_{\mu_1,\mu_2}: Y_\mu \longrightarrow Y_{\mu_1}\otimes Y_{\mu_2}$ such that, for all antidominant coweights $\eta_1,\eta_2$, the following diagram is commutative
\[
\xymatrix{
Y_\mu  \ar[d]_{\iota_{\mu,\eta_1,\eta_2}} \ar[rrr]^{\Delta_{\mu_1,\mu_2}} &&& Y_{\mu_1}\otimes Y_{\mu_2} \ar[d]^{(\iota_{\mu_1,\eta_1,0})\otimes (\iota_{\mu_2,0,\eta_2})}\\
Y_{\mu+ \eta_1+\eta_2} \ar[rrr]_{\Delta_{\mu_1+\eta_1,\mu_2+\eta_2}} &&& Y_{\mu_1+\eta_1}\otimes Y_{\mu_2+\eta_2}
}
\]
\end{Theorem}

\subsection{Truncated shifted Yangians}
\label{section: tsy}

Let $\lambda$ be a dominant coweight, and let $\mu$ be a coweight with $\mu \leq \lambda$.  We can write $\lambda = \sum_i \lambda_i \varpi_i^\vee$ in terms of the fundamental coweights, and $\lambda - \mu = \sum_i m_i \alpha_i^\vee$ in terms of simple coroots.  

We let $\C^\lambda = \prod_i \C^{\lambda_i} / \Sigma_{\lambda_i}$, where $\Sigma_k$ denotes the symmetric group on a set of size $k$.  Then $\C^\lambda$ is the collection of all multisets of sizes $(\lambda_i)_{i \in I}$. A point in $\C^\lambda$ will be denoted $\bR = (R_i)_{i\in I}$, where each $R_i$ is a multiset of size $\lambda_i$, and will be refered as a \emph{set of parameters of coweight} $\lambda$.  Fixing such an $\bR$, we associate the polynomials
$$
p_i(u) = \prod_{c \in R_i} (u-c)
$$

We define elements $A_i^{(r)} \in \C[H_j^{(s)} : j \in I, s > - \langle\mu,\alpha_j\rangle]$, by requiring that the following identity of formal series holds for all $i \in I$:
\begin{equation}
\label{eq: def of A gens}
H_i(u) = \frac{p_i(u) \prod_{j \ne i} \prod_{r=1}^{-a_{ji}} (u-\tfrac{1}{2}d_i a_{ij}- r d_j)^{m_j}}{u^{m_i} (u-d_i)^{m_i} } \frac{ \prod_{j \ne i} \prod_{r=1}^{-a_{ji}} A_j(u-\tfrac{1}{2}d_i a_{ij}- r d_j)}{A_i(u) A_i(u-d_i)}
\end{equation}
where $A_i(u) = 1 + \sum_{r \geq 1 } A_i^{(r)} u^{-r}$. The existence and uniqueness of such elements $A_i^{(r)}$ follows from \cite[Lemma 2.1]{GKLO}.  Note that these elements  depend implicitly on the choice of $\lambda, \mu$ and $\bR$.

Define a ring of difference operators $ \widetilde{\cA} $ as follows.  First, consider the ring having generators $w_{i,r}$ and $\mathsf{u}_{i,r}^{\pm 1}$ for $i \in I$ and $1\leq r \leq m_i$, subject to the relations 
$$
[\mathsf{u}_{i,r}^{\pm 1}, w_{j,s} ] = \pm \delta_{i,j} \delta_{r,s} d_i \mathsf{u}_{i,r}^{\pm 1},
$$
with all other generators commuting.  Then $\widetilde{\cA}$ is its Ore localization given by  introducing inverses $(w_{i,r} - w_{i,s}+ k d_i )^{-1}$ for all $i \in I$, $1\leq r \neq s \leq m_i$ and $k \in \mathbb{Z}$.   We define polynomials with coefficients in $\widetilde{A}$ by
$$
W_i(u) = \prod_{s=1}^{m_i} (u- w_{i,s}), \quad W_{i,r}(u) = \prod_{\substack{s = 1 \\ s\neq r}}^{m_i} ( u -w_{i,s}),
$$
for any $i\in I$ and $1\leq r \leq m_i$.

The shifted Yangian $ Y_\mu $ has a remarkable representation using this algebra of difference operators $ \widetilde{\cA} $.  This representation was first introduced in \cite{GKLO}, further studied in \cite{KWWY}, and generalized and connected with the theory of Coulomb branches in \cite[Appendix B]{BFN} and \cite{NW}.
\begin{Theorem}[\mbox{\cite[Theorem B.15]{BFN}, \cite[Theorem 5.4]{NW}}]  
\label{GKLO homomorphism} Fix any orientation $i\rightarrow j$ of the Dynkin diagram.  
There is a homomorphism of filtered $\C$--algebras $\Phi_\mu^\lambda(\bR): Y_\mu \longrightarrow \widetilde{\cA}$, defined by 
\begin{align*}
A_i(u) & \mapsto u^{-m_i} W_i(u), \\
E_i(u) & \mapsto - d_i^{-1/2}\sum_{r=1}^{m_i} \frac{p_i(w_{i,r}) \prod_{j \rightarrow i} \prod_{s=1}^{-a_{ji}} W_j(w_{i,r}-\tfrac{1}{2}d_i a_{ij}- s d_j )}{(u-w_{i,r}) W_{i,r} (w_{i,r})} \mathsf{u}_{i,r}^{-1}, \\
F_i(u) &  \mapsto d_i^{-1/2}\sum_{r=1}^{m_i} \frac{\prod_{j\leftarrow i} \prod_{s=1}^{-a_{ji}} W_j(w_{i,r}-\tfrac{1}{2} d_i a_{ij} +d_i - s d_j)}{(u-w_{i,r}-d_i) W_{i,r}(w_{i,r})}  \mathsf{u}_{i,r}
\end{align*}
\end{Theorem}

\begin{Def}
The \emph{truncated shifted Yangian} $Y_\mu^\lambda(\bR)$ is defined to be the image of the homomorphism $\Phi_\mu^\lambda(\bR)$.
\end{Def}

\begin{Rem}
Up to isomorphism, $Y_\mu^\lambda(\bR)$ is independent of the choice of orientation of the Dynkin diagram, see \cite[Section 6(viii)]{BFN1}.
\end{Rem}

In the special case when $\lambda = 0$, which will be relevant below, note that there are no parameters $\bR$.  Thus we will simply write $Y_\mu^0$ and $\Phi^0_\mu$.  The most important case for us is the following:

\begin{Lemma}
\label{lemma: simplest coulomb branch of quiver gauge theory}
The map $\Phi_{-\alpha_i^\vee}^0 : Y_{-\alpha_i^\vee} \rightarrow  \widetilde{\cA}$ is surjective.  In particular, $Y_{-\alpha_i^\vee}^0 = \widetilde{\cA}$ has generators $A_i^{(1)}, E_i^{(1)}, F_i^{(1)}$ with relations
$$
[E_i^{(1)}, A_i^{(1)}] = d_i E_i^{(1)}, \quad [F_i^{(1)}, A_i^{(1)}] = - d_i F_i^{(1)}, \quad E_i^{(1)} F_i^{(1)} = F_i^{(1)} E_i^{(1)} = - d_i^{-1}
$$
\end{Lemma}
\begin{proof}
In this case $\widetilde{\cA}$ has generators $w, \mathsf{u}^{\pm 1}$ with relations $[\mathsf{u}^{\pm 1}, w] = \pm d_i \mathsf{u}^{\pm 1}$.  Since $A_i^{(1)} \mapsto - w, E_i^{(1)} \mapsto - d_i^{-1/2} \mathsf{u}^{-1}$ and $F_i^{(1)} \mapsto d_i^{-1/2} \mathsf{u}$, the claim is immediate.
\end{proof}

We may rephrase this result in terms of the ring $D(\C^\times)$ of differential operators on $\C^\times$.  Letting $z$ denote a coordinate on $\C^\times$, the ring $D(\C^\times)$ is generated by $\C[z,z^{-1}]$ along with the operator $\frac{\partial}{\partial z}$.  The following is a straightforward consequence of the previous lemma:

\begin{Lemma}
\label{rem: cotangent bundle of C*}
There is an isomorphism $Y_{-\alpha_i^\vee}^0 \stackrel{\sim}{\longrightarrow} D(\C^\times)$ defined by
$$
A_i^{(1)} \mapsto - d_i z \frac{\partial }{\partial z}, \quad E_i^{(1)} \mapsto z, \quad F_i^{(1)} \mapsto -d_i z^{-1}
$$
\end{Lemma}

For any $Y_\mu^\lambda(\bR)$, note that $A_i(u)$ maps to $u^{-m_i}W_i(u)$, a polynomial in $u^{-1}$ of degree $m_i$.  Therefore $A_i^{(r)}$ is in the kernel of $\Phi_\mu^\lambda$ for $r > m_i$.  As in \cite[Remark B.21]{BFN1}, we conjecture that these elements generate the kernel of $\Phi_\mu^\lambda$.  We would thus get the presentation
$$
Y_\mu^\lambda(\bR) \cong Y_\mu / \langle A_i^{(r)} : i \in I, r  > m_i \rangle
$$
Denote the right-hand side by $\conjectural_\mu^\lambda(\bR)$.  By the above, there is a surjective map $\conjectural_\mu^\lambda(\bR) \twoheadrightarrow Y_\mu^\lambda(\bR)$.

In the appendix to this paper we will prove that $\conjectural_\mu^\lambda(\bR) \cong Y_\mu^\lambda(\bR)$ in type A, see Theorem \ref{appendix thm}.  This generalizes a similar result from \cite{KMWY}, which covers the case when $\mu$ is dominant.

\section{Poisson structure}

\subsection{Filtrations of shifted Yangians} \label{se:filtrations}
We begin with some generalities on filtrations, which can be found in \cite[5.1]{FKPRW}.

Given a $\mathbb{C}$-algebra $A$ with an increasing $\mathbb{Z}$-filtration $F^\bullet A$, $A$ is said to be \emph{almost commutative} if its associated graded algebra $\gr_F A=\bigoplus_{n \in \Z}F^nA/F^{n-1}A$ is commutative.  In this case, $ \gr_F A $ acquires a Poisson algebra structure and we say that $ A $ quantizes the affine Poisson scheme $ \Spec \gr_F A $.  

Explicitly, to compute the Poisson bracket of homogeneous elements $\overline{a} \in F^k A / F^{k-1} A$ and $\overline{b} \in F^\ell A / F^{\ell-1} A$, we choose  lifts $a \in F^k A, b \in F^\ell A$.  Then $[a,b] \in F^{k+\ell-1} A$, and we define the Poisson bracket to be the symbol $\{\overline{a},\overline{b}\} = [a,b] + F^{k+\ell-2} A$. This definition is independent of the choices of $a,b$.  Note in particular that this Poisson bracket has total degree $-1$ with respect to the natural grading on $\gr^F A$.

Let $ \phi : A \rightarrow B $ be a filtered morphism of almost commutative algebras.  We obtain a morphism of commutative Poisson algebras $ \gr \phi : \gr A \rightarrow \gr B $ and thus a morphism of Poisson schemes $ \Spec \gr B \rightarrow \Spec \gr A $.  In this case, we say that $ \phi $ quantizes the morphism $ \Spec \gr B \rightarrow \Spec \gr A $.

The filtration $F^\bullet A$ is said to be \emph{exhaustive} if $A=\bigcup_n F^nA$, is said to be \emph{separated} if $\bigcap_n F^nA=\{0\}$, and is said to \emph{admit an expansion} if there exists a filtered vector space isomorphism $\gr_F A\simeq A$.

Given filtered algebras $F^\bullet A$ and $F^\bullet B$, one can define a filtration on $A\otimes B$ by $F^n(A\otimes B)= \sum_{n=k+l} F^kA \otimes F^l B$. If $F^\bullet A$ and $F^\bullet B$ admit expansions, then $\gr (A\otimes B) \simeq \gr A \otimes \gr B$.

Returning to our setting, given a splitting $\mu=\nu_1+\nu_2$, define a filtration $F_{\nu_1,\nu_2}Y_\mu$ as follows
\begin{equation}
\deg E_\alpha^{(q)} = \langle \nu_1, \alpha \rangle + q, \ \deg F_\beta^{(q)} = \langle \nu_2, \beta \rangle + q, \ \deg H_i^{(p)} = \langle \mu, \alpha_i \rangle + p
\end{equation}
Define the filtered piece $F^k_{\nu_1,\nu_2}Y_\mu$ to be the span of all ordered monomials in PBW variables with total degree at most $k$.

\begin{Proposition}\emph{\cite[Prop~5.7]{FKPRW}}
The filtration $F_{\nu_1,\nu_2}Y_\mu$ is an algebra filtration, is independent of the choice of PBW variables, and is independent of the order of the variables used to form monomials. The algebra $Y_\mu$ is almost commutative.
\end{Proposition}

\begin{Proposition}
\label{prop: filtration properties}
For any splitting $\mu=\nu_1+\nu_2$, $F_{\nu_1,\nu_2}Y_\mu$ is exhaustive, separated, and admits an expansion.
\end{Proposition}
\begin{proof}
These follow from the fact that $Y_\mu$  admits a PBW basis, by Theorem \ref{PBW}. The expansion map is given by $\gr Y_\mu \rightarrow Y_\mu$, $\bar{m}\mapsto m$, where $m$ is a monomial of degree $k$ in PBW variables.
\end{proof}

The following was stated without proof in \cite{FKPRW}.
\begin{Proposition}\label{comultfiltrcomp}
Let $\mu, \mu_k,\nu_k$ ($k=1,2$) be such that $\mu=\mu_1+\mu_2=\nu_1+\nu_2$.
Then the map $\Delta: Y_\mu \rightarrow Y_{\mu_1}\otimes Y_{\mu_2}$ respects the filtrations $F_{\nu_1,\nu_2}Y_\mu$, $F_{\nu_1,\mu_1-\nu_1}Y_{\mu_1}$, and $F_{\mu_2-\nu_2,\nu_2}Y_{\mu_2}$.
\end{Proposition}
\begin{proof}
First, consider the case where $\mu,\mu_1,\mu_2$ are antidominant. Consider the definition of $\Delta$ from Theorem \ref{th:coproductanti}. One can check that $\Delta$ respects filtrations on Levendorskii generators by inspecting degrees. For example, we have that
\[
\Delta(E_j^{(-\langle \mu_1,\alpha_j \rangle +2)})= E_j^{(-\langle \mu_1,\alpha_j\rangle +2)}\otimes 1 + 1\otimes E_j^{(2)} +S_j^{(-\langle \mu_1,\alpha_j\rangle +1)}\otimes E_i^{(1)} -\sum_{\gamma>0}F_\gamma^{(1)}\otimes [E_j^{(1)},E_\gamma^{(1)}].
\]
\begin{align*}
\deg(E_j^{(-\langle \mu_1,\alpha_j \rangle +2)})&= \langle \nu_1,\alpha_j\rangle -\langle \mu_1,\alpha_j\rangle +2=\langle \nu_1-\mu_1,\alpha_j  \rangle + 2, \\
\deg( E_j^{(-\langle \mu_1,\alpha_j\rangle +2)}\otimes 1)&= \langle \nu_1,\alpha_j\rangle -\langle \mu_1,\alpha_j\rangle +2= \langle \nu_1-\mu_1,\alpha_j \rangle + 2, \\
\deg(1\otimes E_j^{(2)})&= \langle \mu_2-\nu_2,\alpha_j \rangle + 2= \langle \nu_1-\mu_1,\alpha_j\rangle +2, \\
\deg(S_j^{(-\langle \mu_1,\alpha_j\rangle +1)}\otimes E_i^{(1)})&= \langle \mu_1,\alpha_j\rangle -\langle \mu_1,\alpha_j\rangle +1 +\langle \nu_1-\mu_1,\alpha_j\rangle+1= \langle \nu_1-\mu_1,\alpha_j\rangle+2, \\
\deg(F_\gamma^{(1)}\otimes [E_j^{(1)},E_\gamma^{(1)}])&= \langle \mu_1-\nu_1,\gamma \rangle +1 + \langle \mu_2-\nu_2, \gamma+\alpha_j\rangle  +1 \\
&= \langle \mu_1-\nu_1,\gamma \rangle +1 + \langle \nu_1-\mu_1, \gamma+\alpha_j\rangle  +1  =\langle \nu_1-\mu_1,\alpha_j\rangle+2.
\end{align*}

Since these filtrations admit expansions by Proposition \ref{prop: filtration properties}, we have that $\gr(Y_{\mu_1}\otimes Y_{\mu_2})\simeq \gr Y_{\mu_1}\otimes \gr Y_{\mu_2}$. Hence, by Theorem \ref{grYmu}, $Y_{\mu_1}\otimes Y_{\mu_2}$ is almost commutative. The higher $E_j^{(r)}, F_j^{(r)}, H_j^{(r)}$ and the PBW variables $E_\beta^{(r)}, F_\beta^{(r)}$ are all obtained from commutators. Thus, $\Delta$ respects their degrees since $Y_{\mu_1}\otimes Y_{\mu_2}$ is almost commutative. To be more precise,
\begin{align*}
\deg \Delta(E_j^{(-\langle \mu_1 ,\alpha_j\rangle +3 )})&= \deg \Delta([S_j^{(-\langle \mu,\alpha_j\rangle +2)},E_j^{(-\langle \mu_1,\alpha_j\rangle +2)}])\\
&= \deg \Delta(S_j^{(-\langle \mu,\alpha_j\rangle +2)}) + \deg \Delta(E_j^{(-\langle \mu_1,\alpha_j\rangle +2)}) - 1 \\
&= 2 + \langle \nu_1-\mu_1,\alpha_j\rangle +2 -1 \\
&= \langle \nu_1-\mu_1,\alpha_j\rangle +3,
\end{align*}
where the second equality uses the almost commutativity of $Y_{\mu_1}\otimes Y_{\mu_2}$. The degrees of the higher $E_j^{(r)}$'s are obtained by induction. Similarly, one can show that $\Delta$ respects the degrees of $F_j^{(r)},H_j^{(r)}, E_\beta^{(r)}, F_\beta^{(r)}$. Hence, $\Delta$ respects filtrations in the case where $\mu,\mu_1,\mu_2$ are antidominant.

For the general case, let $\eta =\eta_1+\eta_2$ be antidominant coweights such that $\mu+\eta, \mu_1+\eta_1$ and $\mu_2+\eta_2$ are all antidominant.  For each $k\in \Z$, we claim that there is a commutative diagram

\[
\xymatrix{
F_{\nu_1,\nu_2}^k Y_\mu \ar[r] \ar[d]_{\iota_{\mu,\eta_1,\eta_2}} & \sum_{k_1+k_2 = k} F_{\nu_1,\mu_1-\nu_1}^{k_1}Y_{\mu_1} \otimes F_{\mu_2-\nu_2,\nu_2}^{k_2}Y_{\mu_2} \ar[d]^{ \iota_{\mu_1,\eta_1,0}\otimes \iota_{\mu_2, 0, \eta_2}} \\
F_{\nu_1+\eta_1,\nu_2+\eta_2}^k Y_{\mu+\eta} \ar[r] & \sum_{k_1+k_2=k}F_{\nu_1+\eta_1,\mu_1-\nu_1}^{k_1}Y_{\mu_1+\eta_1} \otimes F_{\mu_2-\nu_2,\nu_2+\eta_2}^{k_2}Y_{\mu_2+\eta_2}
}
\]
Indeed, consider the corresponding commutative diagram of algebras from Theorem \ref{coproduct}.  In this diagram of algebras  the vertical maps are injective by Proposition \ref{shiftmaps}, and \emph{strictly} respect the filtrations indicated above by the same reasoning as \cite[Lemma 5.5]{FKPRW}.  We have shown above that the bottom map respects filtrations, since $\mu+\eta,\mu_1+\eta_1$ and $\mu_2+\eta_2$ are all antidominant. Together, this proves that the top arrow also respects the filtration, as claimed.
\end{proof}

\subsection{Poisson structure of $\mathcal{W}_\mu$ and quantizations}
\label{section: Poisson structure of Wmu and quantizations}
Again consider a splitting $\mu=\nu_1+\nu_2$ and the filtration $F_{\nu_1,\nu_2}Y_\mu$.  Though the following results were only proven in ADE types in \cite{FKPRW}, the same proofs apply in general. The factors of $d_i^{-1/2}$ in the next theorem are analogous to \cite[Theorem 3.9]{KWWY}.
\begin{Theorem}\label{grYmu} \emph{\cite[Theorem 5.15]{FKPRW}}
	There exists an isomorphism of graded algebras $\gr_{F_{\nu_1,\nu_2}}Y_\mu \simeq \mathbb{C}[\mathcal{W}_\mu]$, defined as follows.  Let $g=uht^\mu u_-\in\mathcal{W}_\mu$, then
	\begin{align*}
	E_j(t)(g)&=d_i^{-1/2} \psi_i(u),\\
	F_j(t)(g)&= d_i^{-1/2}\psi_i^-(u_-), \\
	H_j(t)(g)&=\alpha_j(ht^\mu)
	\end{align*}
	Moreover, the shift morphism $\iota_{\mu,\mu_1,\mu_2}$ of shifted Yangians quantizes the same-named morphism of schemes.
\end{Theorem}

In particular, this implies that $Y_\mu$ is a domain.  As in \cite{FKPRW} we use this theorem to endow $ \cW_\mu $ with a Poisson structure, defined as in Section \ref{se:filtrations}.

\begin{Proposition} \label{antidomgen}\emph{\cite[Lemma~5.17]{FKPRW}}.
	Let $\mu$ be an antidominant coweight. Then $\gr Y_\mu \simeq \mathbb{C}[\mathcal{W}_\mu]$ is generated by $E_i^{(1)}, F_i^{(1)}, H_i^{(-\langle \mu,\alpha_i \rangle +1 )}$, and $H_i^{(- \langle \mu ,\alpha_i \rangle + 2)}$ as a Poisson algebra.
\end{Proposition}

To compare to \cite{KWWY}, recall that $\mathfrak{g}((t^{-1})), t^{-1}\mathfrak{g}[[t^{-1}]]$ and $\mathfrak{g}[t]$ form a Manin triple, which gives rise to a Poisson structure on $G((t^{-1})))$ with $G_1[[t^{-1}]]$ and $G[t]$  as Poisson subgroups.

Comparing \cite[Theorem~3.9]{KWWY} and the proof of \cite[Theorem~5.15]{FKPRW}, one has the following.
\begin{Theorem}
\label{theorem: yangian quantizes loop group}
	The Poisson structure on $\mathcal{W}_0=G_1[[t^{-1}]]$ given by Theorem \emph{\ref{grYmu}} is the same as the structure given by the Manin triple.
\end{Theorem}

\subsection{Poisson structure of $\oW_\mu^\lambda$ and quantizations}
\label{section: Poisson struture on generalized slices}

Since $Y_\mu \twoheadrightarrow Y_\mu^\lambda(\bR)$ is a quotient, there are induced filtrations $F_{\nu_1, \nu_2} Y_\mu^\lambda(\bR)$. 

Combining \cite[Theorem B.28]{BFN}, \cite[Theorem A]{Weekes} and \cite[Theorem 5.8]{NW}, we have:
\begin{Theorem}
\label{grYmula}
Let $\bR$ be any set of parameters of coweight $\lambda$.  There exists an isomorphism of graded algebras $\gr_{F_{\nu_1,\nu_2}}Y^\lambda_\mu(\bR) \simeq \mathbb{C}[\oW_\mu^\lambda]$.    The surjection $Y_\mu \twoheadrightarrow Y_\mu^\lambda(\bR)$ quantizes the closed embedding $\oW_\mu^\lambda \subset \cW_\mu$.
\end{Theorem}

In particular, this endows $\oW_\mu^\lambda$ with a Poisson structure, for which it is a Poisson subvariety of $\cW_\mu$. 

\begin{Rem}
Note that Theorems \ref{grYmu}, \ref{grYmula} apply for any splitting $\mu = \nu_1+\nu_2$, and thus induce gradings on $\C[\cW_\mu]$ and $\C[\oW_\mu^\lambda]$.  These gradings correspond to the $\C^\times$-actions $\kappa^{\nu_1, \nu_2}$ on $\cW_\mu$ (easily seen to preserve $\oW_\mu^\lambda$),  as explained in \cite[Theorem 5.15]{FKPRW}.
\end{Rem}

The symplectic leaves for this Poisson structure are known explicitly: for any dominant $\lambda \geq \mu$, define the (dense) open subvariety
$$
\cW_\mu^\lambda = \cW_\mu \cap G[t] t^\lambda G[t] \ \subseteq \ \oW_\mu^\lambda
$$
There is a decomposition
\begin{equation}
\label{eq: symplectic leaves}
\oW_\mu^\lambda = \bigcup_{\substack{\mu \leq \nu \leq \lambda, \\ \nu \text{ dominant}}} \cW^\nu_\mu
\end{equation}
\begin{Theorem}[\mbox{\cite{MW}}]
Equation (\ref{eq: symplectic leaves}) gives the decomposition into symplectic leaves.
\end{Theorem}

Recall the explicitly presented algebra $\conjectural_\mu^\lambda(\bR)$ from section \ref{section: tsy}.   The following result generalizes \cite[Theorem 4.8]{KWWY}:
\begin{Proposition}
\label{prop: conjectural yangian quantizes scheme}
$\conjectural_\mu^\lambda(\bR)$ quantizes a scheme supported on $\oW_\mu^\lambda$.
\end{Proposition}
\begin{proof}
On the one hand, since
$$
\C[ \cW_\mu] \cong \gr Y_\mu \twoheadrightarrow \gr \conjectural_\mu^\lambda(\bR) \twoheadrightarrow \gr Y_\mu^\lambda(\bR) \cong \C[\oW_\mu^\lambda],
$$
we see that $\conjectural_\mu^\lambda(\bR)$ quantizes a closed subscheme $M \subset \cW_\mu$ with $\oW_\mu^\lambda \subseteq M$.  Denote the corresponding reduced scheme by $M_{red}$.  Then $\overline{\cW}_\mu^\lambda \subseteq M_{red}$.

On the other hand, the kernel of $\C[\cW_\mu] \twoheadrightarrow \gr \conjectural_\mu^\lambda(\bR)$ certainly contains the Poisson ideal $J_\mu^\lambda$ generated by the classical limits of the elements $A_i^{(r)}$ for $r>m_i$. In Corollary \ref{prop: intermediate KMW} we will prove that the radical $\sqrt{J_\mu^\lambda}$ is the defining ideal of $\overline{\cW}_\mu^\lambda \subset \cW_\mu$. Therefore we obtain the opposite inequality $M_{red} \subseteq \overline{\cW}_\mu^\lambda$, which proves the claim.

\end{proof}

\subsection{Multiplication is Poisson}

\begin{Proposition}\label{shiftpoi}
	The shift maps are Poisson and are quantized by  the shift morphisms of Proposition \ref{shiftmaps}. If $\lambda$ is dominant such that $\lambda \geq \mu$ and $\lambda+\mu_1+\mu_2$ is dominant, then the shift map $\iota_{\mu,\mu_1,\mu_2}$ restricts to a map $\overline{\mathcal{W}}^{\lambda +\mu_1 +\mu_2}_{\mu+\mu_1+\mu_2} \longrightarrow \overline{\mathcal{W}}^{\lambda}_\mu$. The restricted map is Poisson and birational.
\end{Proposition}
\begin{proof}
	The first claim follows from Theorem \ref{grYmu}, since the shift maps quantize to shift homomorphisms for shifted Yangians. If $g \in \overline{G[t] t^{\lambda+\mu_1+\mu_2} G[t]}$, then 
$$
t^{-\mu_1} g t^{-\mu_2} \in \overline{G[t] t^{-\mu_1} G[t] } \cdot \overline{G[t] t^{\lambda+\mu_1+\mu_2} G[t]} \cdot\overline{G[t] t^{-\mu_2} G[t] } 
$$
For any dominant $\lambda_1,\lambda_2$, it is well-known that
$$
\overline{G[t] t^{\lambda_1} G[t]} \cdot \overline{G[t] t^{\lambda_2} G[t]} \ \subseteq \ \overline{G[t] t^{\lambda_1+\lambda_2} G[t]}
$$
Since $-\mu_1,\lambda_1+\mu_1+\mu_2$ and $-\mu_2$ are all dominant, it follows that $t^{-\mu_1} g t^{-\mu_2} \in \overline{G[t] t^\lambda G[t]}$.  Since $\pi(t^{\mu_1} g t^{\mu_2})$ is defined by left (resp.~right) multiplying by $U[t]$ (resp.~$U_-[t]$), we see that $\pi(t^{\mu_1} g t^{\mu_2}) \in \overline{G[t] t^{\lambda} G[t]}$.  It follows that $\iota_{\mu,\mu_1,\mu_2}$ restricts to a map $\overline{\mathcal{W}}^{\lambda +\mu_1 +\mu_2}_{\mu+\mu_1+\mu_2} \longrightarrow \overline{\mathcal{W}}^{\lambda}_\mu$ as claimed.   % 	\vskip2ex
	Finally, the birationality follows from \cite[Rem.~3.11]{BFN}.
\end{proof}

Next, we want to show the multiplication is Poisson. We will use the following lemma several times.

\begin{Lemma}\label{birpoi}
	Let $X_i, Y_i$ ($i=1,2$) be irreducible affine Poisson varieties. Suppose that we have a commutative diagram
	\[
	\xymatrix{
		X_1 \ar[d]_{i} \ar[rr]^{f_1}&& Y_1 \ar[d]^{j}\\
		X_2 \ar[rr]_{f_2}&& Y_2
	}
	\]
	such that the vertical maps are birational and Poisson. Then:
	\begin{enumerate}
	\item[(a)] If $f_1$ is Poisson, then $f_2$ is Poisson.
	
	\item[(b)] Suppose that there exist open subsets $V_1 \subseteq Y_1$ and $V_2 \subseteq Y_2$ which are isomorphic under $j$, such that $f_1^{-1}(V_1) \neq \emptyset$. If $f_2$ is Poisson, then $f_1$ is Poisson.
	\end{enumerate}
\end{Lemma}
In particular, note that the assumption on open subsets in part (b) is satisfied if $f_1$ is dominant.
\begin{proof}
	Since $i$ is birational, there exist open sets $U_1\subseteq X_1$ and $U_2\subseteq X_2$ such that $U_1\simeq U_2$. By commutativity of the diagram, we see that $f_2|_{U_2}= j\circ f_1\circ (i^{-1}|_{U_2})$. If $f_1$ is Poisson, we see that $f_2|_{U_2}$ is Poisson. Consider the following commutative diagram
	\[
	\xymatrix{
		\mathbb{C}[Y_2] \ar[rrd]_{f_2|_{U_2}^*}\ar[rr]^{f_2^*}&& \mathbb{C}[X_2]  \ar[d]\\
		&& \mathbb{C}[U_2]
	}
	\]
	We see that the vertical map is injective and the diagonal map is Poisson. Thus, $f_2$ is also Poisson. This proves (a).
	
	For (b), let $U_1'=f_1^{-1}(V_1)$, which is non-empty by our assumption. By commutativity of the diagram, $(f_2\circ i)(U_1')\subseteq V_2$. We see that $f_1|_{U'_1}=j^{-1}\circ f_2 \circ (i|_{U_1'})$. Thus, if $f_2$ is Poisson, so is $f_1|_{U'_1}$. Therefore, $f_1$ is also Poisson by the same reasoning as above.
\end{proof}

\begin{Proposition}
\label{prop: mult map restricts to slices}
	The multiplication map $m_{\mu_1,\mu_2}: \mathcal{W}_{\mu_1} \times \mathcal{W}_{\mu_2} \longrightarrow \mathcal{W}_{\mu_1+\mu_2}$ restricts to a map $\overline{\mathcal{W}}_{\mu_1}^{\lambda_1} \times \overline{\mathcal{W}}_{\mu_2}^{\lambda_2} \longrightarrow \overline{\mathcal{W}}_{\mu_1+\mu_2}^{\lambda_1+\lambda_2}$. Moreover, the restricted map is dominant and Poisson.
\end{Proposition}
\begin{proof}
	The first claim follows from comparing the constructions of \cite[2(vi) and 2(xi)]{BFN}, see also \cite[Section 5.9]{FKPRW} or \cite[Section 5.2.2]{KrylovP}. This restricted map is dominant by \cite[Prop.~5.7]{KrylovP}. 
	
	To prove that the restricted map is Poisson, first consider the case where $\mu_1=\mu_2=0$. We know that $\mathcal{W}_0=G_1[[t^{-1}]]$ is a Poisson algebraic group. The map $m_{0,0}$ is precisely the group multiplication in $G_1[[t^{-1}]]$. Hence, it is Poisson, and so are its restrictions.
	
	Next, suppose that $\mu_1,\mu_2$ are dominant. If $\lambda_1 \geq \mu_1$ and $\lambda_2 \geq \mu_2$, consider $\nu_1=-\mu_1$, $\nu_2=-\mu_2$. We have the following slice version of Lemma \ref{shiftmultcomp}.
	\[
	\xymatrix{
		\overline{\mathcal{W}}^{\lambda_1-\mu_1}_0 \times \overline{\mathcal{W}}^{\lambda_2-\mu_2}_0 \ar[rr] \ar[d]_{\iota_{\mu_1,-\mu_1,0}\times \iota_{\mu_2,0,-\mu_2}} & & \overline{\mathcal{W}}^{\lambda_1+\lambda_2 -\mu_1-\mu_2}_0 \ar[d]^{\iota_{\mu_1+\mu_2,-\mu_1,-\mu_2}} \\
		\overline{\mathcal{W}}^{\lambda_1}_{\mu_1} \times \overline{\mathcal{W}}^{\lambda_2}_{\mu_2} \ar[rr] & & \overline{\mathcal{W}}^{\lambda_1+\lambda_2}_{\mu_1+\mu_2}
	}
	\]
	By Proposition \ref{shiftpoi}, the two vertical arrows are Poisson and birational. Since the top arrow is Poisson, by part (a) of Lemma \ref{birpoi}, the bottom arrow is also Poisson, proving this case.
	
	Finally, suppose that $\mu_1$ and $\mu_2$ are arbitrary. We can choose  choose $\nu_1,\nu_2$ antidominant such that $\mu_1-\nu_1, \mu_2-\nu_2$ are dominant. Then we have another slice version of Lemma \ref{shiftmultcomp}.
	\[
	\xymatrix{
		\overline{\mathcal{W}}^{\lambda_1}_{\mu_1} \times \overline{\mathcal{W}}^{\lambda_2}_{\mu_2} \ar[rr] \ar[d]_{\iota_{\mu_1-\nu_1,\nu_1,0}\times \iota_{\mu_2-\nu_2,0,\nu_2}} & & \overline{\mathcal{W}}^{\lambda_1+\lambda_2 }_{\mu_1+\mu_2} \ar[d]^{\iota_{\mu_1+\mu_2-\nu_1-\nu_2,\nu_1,\nu_2}} \\
		\overline{\mathcal{W}}^{\lambda_1-\nu_1}_{\mu_1-\nu_1} \times \overline{\mathcal{W}}^{\lambda_2-\nu_2}_{\mu_2-\nu_2} \ar[rr] & & \overline{\mathcal{W}}^{\lambda_1+\lambda_2-\nu_1-\nu_2}_{\mu_1+\mu_2-\nu_1-\nu_2}
	}
	\]
	The bottom arrow is Poisson by our previous case. Since the top arow is dominant, part (b) of Lemma \ref{birpoi} applies, so the top arrow is also Poisson.
\end{proof}

The following result was conjectured in \cite{FKPRW} (Conjecture~5.20).

\begin{Theorem}\label{multpoi}For any $ \mu_1, \mu_2$, the multiplication map $ \cW_{\mu_1}\times \cW_{\mu_2} \rightarrow \cW_{\mu_1+\mu_2}$ is Poisson and is quantized by the comultiplication $ \Delta : Y_{\mu_1 + \mu_2} \rightarrow Y_{\mu_1} \otimes Y_{\mu_2}$.
\end{Theorem}
\begin{proof}
Let $\Delta'$ be the algebra morphism $ \Delta' : \C[\cW_{\mu_1 + \mu_2}] \rightarrow \C[\cW_{\mu_1} \otimes \cW_{\mu_2}]$ coming from the multiplication $ \cW_{\mu_1}\times \cW_{\mu_2} \rightarrow \cW_{\mu_1+\mu_2}$.	Let $f,g\in \mathbb{C}[\mathcal{W}_{\mu_1+\mu_2}]$.  We need to show that $h:=\Delta'(\{f,g\})-\{\Delta'(f),\Delta'(g)\}=0$. By Proposition \ref{denseunion}, it suffices to show that the restriction of $h$ on each $\overline{\mathcal{W}}_{\mu_1}^{\lambda_1} \times \overline{\mathcal{W}}_{\mu_2}^{\lambda_2}$ is zero.
	
	Let $I$ be the ideal of $\overline{\mathcal{W}}_{\mu_1+\mu_2}^{\lambda_1+\lambda_2} \subset \cW_{\mu_1+\mu_2}$ and let $J$ be the ideal of $\overline{\mathcal{W}}_{\mu_1}^{\lambda_1} \times \overline{\mathcal{W}}_{\mu_2}^{\lambda_2} \subset \cW_{\mu_1} \times \cW_{\mu_2}$. From Theorem \ref{grYmula}, we see that these are Poisson ideals. Since $\Delta'(I)\subseteq J$ by Proposition \ref{prop: mult map restricts to slices}, and since the restriction maps are Poisson,
	\begin{align*}
	0+J &=\Delta'(\{f+I, g+I\})-\{\Delta'(f+I),\Delta'(g+I)\} + J \\
	&= \Delta'(\{f,g\})-\{\Delta'(f),\Delta'(g)\} + J \\
	&=h+J.
	\end{align*}
	Therefore, the multiplication is Poisson.
	
	Now that we have established that the multiplication is Poisson, the second statement follows from \cite[Prop 5.21]{FKPRW}.
\end{proof}

\section{The Hamiltonian reduction}

\subsection{Some notation}
Fix $ i \in I $.  We will need some notation related to this choice of simple root, some of which we will recall from Section \ref{section:notation}.

Recall the group homomorphism $ \psi_i : U \rightarrow \Ga $ and let $ U^i \subset U $ be its kernel.

Recall also that we extended $ \psi_i $ to a map $ \psi_i : U T U_- \rightarrow \C $ by $\psi_i(u t v) = \psi_i(u)$.  This leads to $ \psi_i : \cW_\mu \rightarrow t^{-1}\C[[t^{-1}]]$.  We will need to examine the Fourier coefficients of this map, which we will denote $ \psi_i^{(k)} : \cW_\mu \rightarrow \C $ for $ k = 1, 2, \dots $.  We will be particularly interested in the first coefficient, and define
\begin{equation}
\Phi_i := d_i^{-1/2} \psi_i^{(1)}
\end{equation}
Note that $\Phi_i$ is the classical limit of $E_i^{(1)} \in Y_\mu$, by Theorem \ref{grYmu}.

Finally, we need a special subgroup in which we allow Laurent series in $U^i$, but only principal parts in the root subgroup corresponding to $i$: $$ U((t^{-1}))^i_1 := x_i(t^{-1}\C[[t^{-1}]])U^i((t^{-1})) =  U^i((t^{-1})) x_i(t^{-1}\C[[t^{-1}]])$$
Here, recall that $x_i: \mathbb{G}_a \rightarrow U$ exponentiates the Chevalley generator $e_i$, which induces a map $x_i : \C((t^{-1})) \rightarrow U((t^{-1}))$. Note that $U((t^{-1}))_1^i$ is the inverse image of $t^{-1}\C[[t^{-1}]]$ under $\psi_i :U((t^{-1}))\rightarrow \C((t^{-1}))$ (induced by $\psi_i: U\rightarrow \mathbb{G}_a$).

Recall that for each $i\in I$, there is a homomorphism $\tau_i: SL_2\rightarrow G$. We record the following facts for future use.
\begin{Proposition} \label{pr:commuteUi}
	\begin{enumerate}
\item If $ u \in U^i$ and $ g \in SL_2 $, then $ \tau_i(g) u \tau_i(g)^{-1} \in U^i$.
\item If $ u \in U_1^i[[t^{-1}]] $ and $ g \in SL_2[[t^{-1}]] $, then $  \tau_i(g) u \tau_i(g)^{-1} \in U^i_1[[t^{-1}]]$.
\item If $ u \in U^i[t] $ and $ g \in SL_2[t] $, then $  \tau_i(g) u \tau_i(g)^{-1} \in U^i[t]$.
\end{enumerate}
\end{Proposition}

\subsection{An action of the additive group}
Let $\mu$ be any coweight.
\begin{Lemma} \label{le:neq1}
Let $g\in \mathcal{W}_\mu$, $a\in \mathbb{G}_a$. Then $x_{-i}(a)g \in U_1[[t^{-1}]] T_1[[t^{-1}]]t^\mu U_-((t^{-1}))$.
\end{Lemma}
\begin{proof}
Write $$g= x_i(p) u'ht^\mu u_-, \quad p \in t^{-1}\C[[t^{-1}]], u' \in U^i_1[[t^{-1}]], h \in T_1[[t^{-1}]], u_- \in U_{-,1}[[t^{-1}]]$$
Since $ a\in \C, p \in t^{-1}\C[[t^{-1}]]$, we see that $1 + ap \in 1 + t^{-1}\C[[t^{-1}]] $, and
$$
\begin{bmatrix} 1 & 0 \\ a & 1 \end{bmatrix} \begin{bmatrix} 1 & p \\ 0 & 1 \end{bmatrix} = \begin{bmatrix} 1 & p(1 + ap)^{-1} \\ 0 & 1 \end{bmatrix} \begin{bmatrix} (1+ap)^{-1} & 0 \\ a & 1 + ap \end{bmatrix}
$$
Thus,
$$
x_{-i}(a) g = x_i(p(1+ap)^{-1}) \bigl( \tau_i(b_-) u' \tau_i(b_-)^{-1} \bigr) \tau_i(b_-) h t^\mu u_-
$$
where $ b_- = \begin{bmatrix} (1+ap)^{-1} & 0 \\ a & 1 + ap \end{bmatrix}$.  By Proposition \ref{pr:commuteUi}.(2), $ \tau_i(b_-) u' \tau_i(b_-)^{-1} \in U^i_1[[t^{-1}]] $ and the result follows.
\end{proof}

\begin{Proposition}\label{actionWmu}
The expression $a\cdot g= \pi\big(x_{-i}(- d_i^{1/2}a)g\big)$ defines an action of $\mathbb{G}_a$ on $\mathcal{W}_\mu$.    It restricts to an action on each $\oW_\mu^\lambda$.
\end{Proposition}
\begin{proof}
 Let $a_1, a_2\in \mathbb{C}$, let $g=uht^\mu u_- \in \mathcal{W}_\mu$.  By Lemma \ref{le:neq1}, $ \pi\big(x_{-i}(-d_i^{1/2}a_2)g\big) = x_{-i}(-d_i^{1/2}a_2) g n_- $ for some $ n_- \in U_-[t] $. Thus
 $$
 a_1 \cdot (a_2 \cdot g) = \pi\big(x_{-i}(-d_i^{1/2}a_1) x_{-i}(-d_i^{1/2}a_2) gn_-) = \pi\big(x_{-i}(-d_i^{1/2}(a_1+a_2)) gn_-\big) 
$$
$$
= \pi\big(x_{-i}(-d_i^{1/2}(a_1 + a_2)) g\big) = (a_1+a_2)\cdot g
$$
as desired.

For the final claim, let $g \in \oW_\mu^\lambda$ and write $a\cdot g = x_{-i}(-d_i^{1/2} a ) g n_-$ as above.  Since $\overline{G[t] t^\lambda G[t]}$ is preserved by left (or right) multiplication by $G[t]$, we see that $a\cdot g \in \overline{G[t] t^\lambda G[t]}$.  Since we already know that $a\cdot g \in \cW_\mu$, this proves that $a\cdot g \in \oW_\mu^\lambda$.
\end{proof}

Later we will see that this action is Hamiltonian with moment map $ \Phi_i $.

\subsection{Relating $\mathcal{W}_\mu$ and $\mathcal{W}_{\mu+\alpha_i^\vee}$} \label{reduction}
Let $ \Phi_i^{-1}(\mathbb{C}^\times)_\mu $ denote the preimage of $ \C^\times $ under $ \Phi_i : \cW_\mu \rightarrow \C $, and similarly $\Phi_i^{-1}(\Cx)_\mu^\lambda$ denote the preimage under the restriction of $\Phi_i$ to $\overline{\cW}_\mu^\lambda$.  The goal of this section is to construct an isomorphism $\Phi_i^{-1}(\mathbb{C}^\times)_\mu \simeq \overline{\mathcal{W}}_{-\alpha_i^\vee}^0 \times \mathcal{W}_{\mu+\alpha_i^\vee}$ using the multiplication map $m = m_{-\alpha_i^\vee, \mu+\alpha_i^\vee}$ from section \ref{multsection}.

We begin by examining $ \oW_{-\alpha_i^\vee}^0$. 

\begin{Lemma}
\label{le:W0alphai}
There is an isomorphism $T^* \Cx = \Cx \times \C \rightarrow \oW^0_{-\alpha_i^\vee}$ given by 
$$
(b,c) \mapsto r_i(b,c) := \tau_i \left( \begin{bmatrix} 0 & b \\ -b^{-1} & t-c \end{bmatrix}\right)
$$
Moreover $ \psi_i(r_i(b,c)) = b(t-c)^{-1} = bt^{-1} + bc t^{-2} + \dots $.

Finally, the Poisson structure on $ \oW^0_{-\alpha_i^\vee} $ is given by $ \{c, b\} = d_i b $.
\end{Lemma}
\begin{proof}
Inside $SL_2((t^{-1}))$, we may write:
$$
\begin{bmatrix} 0 & b \\ -b^{-1} & t-c \end{bmatrix} = \begin{bmatrix} 1 &  b(t-c)^{-1} \\ 0 & 1 \end{bmatrix} \begin{bmatrix} t(t-c)^{-1} & 0 \\ 0 & (t-c)t^{-1} \end{bmatrix} \begin{bmatrix} t^{-1} & 0 \\ 0 & t \end{bmatrix} \begin{bmatrix} 1 & 0 \\ -b^{-1} (t-c)^{-1} & 1 \end{bmatrix}
$$
The left side is regular in $t$, while the right side lies in the scheme $\cW_{-\alpha^\vee}$ for $SL_2$. Thus the above matrix is an element of $\oW_{-\alpha^\vee}^0$ for $SL_2$.  Applying $\tau_i$, we see  that $r_i(b,c) \in \oW_{-\alpha_i^\vee}^0$.  We  also see that $\psi_i (r_i(b,c)) = b(t-c)^{-1}$.

By Theorem \ref{grYmu}, the classical limit of $E_i^{(1)}$ evaluates to $d_i^{-1/2} b$ on $r_i(b,c)$.  Similarly $A_i^{(1)}$ evaluates to $-c$, see Appendix \ref{appendix: notation}.  Finally, the classical limit of Lemma \ref{lemma: simplest coulomb branch of quiver gauge theory} implies that $r_i$ is an isomorphism, and also implies the claimed Poisson structure.
\end{proof}

\begin{Lemma}
\label{lemma: hamiltonian action on W0alphai}
The action of $\mathbb{G}_a$ on $\oW^0_{-\alpha_i^\vee}$ from Proposition \ref{actionWmu} is given explicitly in the coordinates of the previous lemma by
$$
a \cdot r_i(b,c) = r_i(b, c+d_i^{1/2} ab)
$$
Moreover:
\begin{enumerate}
\item This action of $\mathbb{G}_a$ is Hamiltonian with moment map $\Phi_i$.

\item The action of $\mathbb{G}_a$ on $\Phi_i^{-1}(1)_{-\alpha_i^\vee}^0 \subset \oW_{-\alpha_i^\vee}^0$ is free and transitive.
\end{enumerate}
\end{Lemma}
\begin{proof}
To prove the first claim, we compute the action of $a \in \mathbb{G}_a$:
\begin{align*}
a \cdot r_i(b,c) &= \tau_i\left( \begin{bmatrix} 1 & 0 \\ -d_i^{1/2}  as & 1 \end{bmatrix} \cdot \begin{bmatrix} 0 & b \\ -b^{-1} & t-c \end{bmatrix} \right)  \\ &= \tau_i \left( \begin{bmatrix} 0 & b \\ -b^{-1} & t-c- d_i^{1/2} a b \end{bmatrix}\right)
\\ &= r_i(b, c+d_i^{1/2} a b),
\end{align*}

Now, the infinitesimal action of $s \in \operatorname{Lie}\mathbb{G}_a$  on a function $f \in \C[\oW^0_{-\alpha_i^\vee}]$ is given by 
$$
(s \cdot f) (g) = \frac{d}{d\epsilon} f \big( (- \epsilon s)\cdot g\big) \big|_{\epsilon = 0}
$$
Writing $g = r_i(b,c)$, the computation above shows that
$$
(-\epsilon s) \cdot r_i(b,c) = r_i(b, c-d_i^{1/2} \epsilon sb) 
$$
Therefore $s$ acts by the vector field $-  d_i^{1/2} s b \frac{\partial}{\partial c}$. Using the previous lemma, this is equal to the Hamiltonian vector field
$\{d_i^{-1/2}s b, - \} = \{ s\Phi_i, -\}$.  This shows that $\Phi_i$ is a moment map for this action, proving (1).

Finally  we have $\Phi_i(r_i(b,c)) = d_i^{-1/2} b$, so the fiber $\Phi_i^{-1}(1)_{-\alpha_i^\vee}^0$ corresponds to the locus where $b = d_i^{1/2}$.  The $\mathbb{G}_a$ action on this fiber is given by
$$
a\cdot r_i( d_i^{1/2}, c) = r_i(d_i^{1/2},  c+d_i a)
$$
for $a\in \mathbb{G}_a$ and $c\in \C$, which is clearly free and transitive.  This proves (2).
\end{proof}

\begin{Rem}
This lemma explains our choice of $\Ga$ action in Proposition \ref{actionWmu}. But this choice is not so important: it is not hard to see that for any $\kappa \in \Cx$, the  $\Ga$ action defined similarly by $a\cdot g = \pi( x_{-i}( \kappa a) g)$ is also Hamiltonian, and that the corresponding Hamiltonian reductions are all isomorphic.
\end{Rem}

We now begin examining the multiplication map.
\begin{Lemma} \label{le:g1g2}
Let $g_1\in \overline{\mathcal{W}}_{-\alpha_i^\vee}^0$ and $g_2\in \mathcal{W}_{\mu+\alpha_i^\vee}$. Then
\begin{enumerate}
\item $ g_1 g_2 \in U((t^{-1}))^i_1T_1[[t^{-1}]] t^\mu U_-((t^{-1}))$
\item There exists $n \in U^i[t], n_- \in U_-[t]$ such that $ \pi(g_1 g_2) = n g_1 g_2 n_- $.
\item $\psi^{(k)}_i(\pi(g_1 g_2)) = \psi^{(k)}_i(g_1)$ for $ k = 1, 2 $, and in particular $\Phi_i( m(g_1, g_2)) = \Phi_i(g_1)$.
\end{enumerate}
\end{Lemma}
\begin{proof}
Let $g_1 = u_1t^{-\alpha_i^\vee}h_1 u_{-,1} \in \overline{\mathcal{W}}_{-\alpha_i^\vee}^0$, $g_2 = u_2h_2t^{\mu +\alpha_i^\vee}u_{-,2} \in \mathcal{W}_{\mu+\alpha_i^\vee}$.  We apply Gauss decomposition 
$$h_1 u_{-,1} u_2 h_2=u_3 h_3 u_{-,3}$$ 
where $u_i\in U_1[[t^{-1}]]$, $h_i\in T_1[[t^{-1}]]$, and $u_{-,i}\in U_{-,1}[[t^{-1}]]$. Then we have
\[
g_1 g_2 = u_1(t^{-\alpha_i^\vee}u_3t^{\alpha_i^\vee}) h_3 t^{\mu} (t^{-(\mu+\alpha_i^\vee)}u_{-,3} t^{\mu+\alpha_i^\vee})u_{-,2}.
\]
Now $ \psi_i(u_1(t^{-\alpha_i^\vee}u_3t^{\alpha_i^\vee})) = \psi_i(u_1) + t^{-2} \psi_i(u_3) \in t^{-1}\C[[t^{-1}]] $ implying (3).  This also implies that $ g_1 g_2 \in U((t^{-1}))^i_1 T_1[[t^{-1}]] t^\mu U_-((t^{-1}))$, and hence (1) follows.  For (2), we observe that for $ u \in U((t{-1}))^i_1 $, there exists $ n \in U^i[t] $ such that $ nu \in U_1[[t^{-1}]]$.
\end{proof}
The function $\Phi_i \in \C[\oW^0_{-\alpha_i^\vee}]$ is a unit, thus by part (3) above we have  
$$
m:\overline{\mathcal{W}}_{-\alpha_i^\vee}^0\times \mathcal{W}_{\mu +\alpha_i^\vee} \longrightarrow \Phi_i^{-1}(\mathbb{C}^\times)_\mu \subset \cW_{\mu}
$$

 Next, we wish to define the inverse to $m$. First, consider the map
\[
\xi_i: \Phi_i^{-1}(\mathbb{C}^\times)_\mu \longrightarrow \overline{\mathcal{W}}_{-\alpha_i^\vee}^0, \quad g \mapsto r_i\big(\psi_i^{(1)}(g), {\psi_i^{(1)}(g)}^{-1}\psi_i^{(2)}(g)\big)  .
\]

\begin{Lemma} \label{le:impsi}
For any $ g\in \Phi_i^{-1}(\Cx)_\mu $, we have $ \xi_i(g)^{-1} g \in U((t^{-1}))_1^i T_1[[t^{-1}]] t^{\mu + \alpha_i^\vee} U_-((t^{-1}))$.
\end{Lemma}

\begin{proof}
Let $ p = \psi_i(g) $ and write $g= x_i(p) u' ht^\mu u_- \in \Phi_i^{-1}(\mathbb{C}^\times)$ with $ u' \in U^i_{1}[[t^{-1}]], h \in T_1[[t^{-1}]], u_- \in U_{-, 1}[[t^{-1}]]$.
Then
$$ \xi_i(g)^{-1} = r_i(p^{(1)}, {p^{(1)}}^{-1} p^{(2)})^{-1} = \tau_i \left(\begin{bmatrix} t - {p^{(1)}}^{-1} p^{(2)} & -p^{(1)} \\ {p^{(1)}}^{-1} & 0 \end{bmatrix} \right)$$

In $ SL_2((t^{-1}))$, we have
$$
 \begin{bmatrix} t - {p^{(1)}}^{-1} p^{(2)} & -p^{(1)} \\ {p^{(1)}}^{-1} & 0 \end{bmatrix}  \begin{bmatrix} 1 & p \\ 0 & 1 \end{bmatrix} =  \begin{bmatrix} 1 & q \\ 0 & 1 \end{bmatrix} \begin{bmatrix} p^{(1)} p^{-1} & 0 \\ {p^{(1)}}^{-1} & {p^{(1)}}^{-1} p  \end{bmatrix}
$$
where $ q = p^{(1)}t - p^{(2)} - (p^{(1)})^2 p^{-1} $.
Note that $ q \in t^{-1}\C[[t^{-1}]] $ and that $ p^{(1)} p^{-1} \in t + \C[[t^{-1}]] $.

Next we compute that 
$$
 b_- = \begin{bmatrix} p^{(1)} p^{-1} & 0 \\ {p^{(1)}}^{-1} & {p^{(1)}}^{-1} p  \end{bmatrix} = \begin{bmatrix} t & 0 \\ 0 & t^{-1} \end{bmatrix} \begin{bmatrix} t^{-1} p^{(1)} p^{-1} & 0 \\ 0 & (t^{-1}  p^{(1)} p^{-1})^{-1} \end{bmatrix} \begin{bmatrix} 1 & 0 \\ t^{-1} p^{-1} & 1 \end{bmatrix}
 $$
We see from this that $\tau_i(b_-) \in t^{\alpha_i^\vee} T_1[[t^{-1}]] U_-((t^{-1}))$.

%  Note that $ b_- \in t^{\alpha^\vee} T_1[[t^{-1}]] U_-((t^{-1}))$ in .  

Thus
\begin{align*} \xi_i(g)^{-1} g &= \xi_i(g)^{-1} x_i(p) u' ht^\mu u_- \\
&= x_i(q) \tau_i(b_-) u' h t^\mu u_- = x_i(q) \bigl( \tau_i(b_-) u' \tau_i(b_-)^{-1} \bigr) \tau_i(b_-) h t^\mu u_-
\end{align*}
By Proposition \ref{pr:commuteUi}.(1), we see that $ \tau_i(b_-) u' \tau_i(b_-)^{-1}\in U^i((t^{-1})) $.  Since $\tau_i(b_-) \in t^{\alpha_i^\vee} T_1[[t^{-1}]] U_-((t^{-1}))$, observe also that $\tau_i(b_-) h t^\mu u_-  \in  t^{\mu+\alpha_i^\vee} T_1[[t^{-1}]] U_-((t^{-1}))$. So the result follows.
\end{proof}

We define a candidate for the inverse of $ m $ as follows,
\begin{align*}
f: \Phi_i^{-1}(\mathbb{C}^\times)_\mu \longrightarrow \overline{\mathcal{W}}_{-\alpha_i^\vee}^0 \times \cW_{\mu+\alpha_i^\vee}, \quad g\mapsto (\xi_i(g), \pi(\xi_i(g)^{-1}g)).
\end{align*}

\begin{Lemma} \label{le:xi}
Let $ g_1 \in \oW^0_{-\alpha_i^\vee}, g_2 \in \cW_{\mu + \alpha_i^\vee} $.  Then $ \xi_i(\pi(g_1g_2)) = g_1 $.

\end{Lemma}
\begin{proof}
By Lemma \ref{le:g1g2}.(3), we see that $ \psi_i^{(k)}(\pi(g_1 g_2)) = \psi_i^{(k)}(g_1) $ for $ k = 1, 2$.  By Lemma \ref{le:W0alphai}, these two Fourier coefficients determine $ g_1$.
\end{proof}

\begin{Theorem}\label{psiminverses}
The maps $f$ and $m$ are inverses of each other.
\end{Theorem}
\begin{proof}
First, we show that $m \circ f =\Id$.

Let $ g \in \Phi_i^{-1}(\Cx)_\mu $.  From Lemma \ref{le:impsi}, there exist $ n \in U^i[t] $ and $ n_- \in U_-[t] $ such that  $ m (f(g)) = \pi(\xi_i(g) n \xi_i(g)^{-1} g n_-) $.  Since $ \xi_i(g) \in \tau_i (SL_2[t])$, by Proposition \ref{pr:commuteUi}(3) we have that $ \xi_i(g) n \xi_i(g)^{-1} \in U^i[t] $, so $ \pi(\xi_i(g) n \xi_i(g)^{-1} g n_-) = \pi(g) = g $ as desired.  

For the reverse direction, let $ g_1 \in \oW^0_{-\alpha_i^\vee}, g_2 \in \cW_{\mu + \alpha_i^\vee} $. By Lemma \ref{le:xi}, we see that
$$
f(m(g_1, g_2)) = (g_1, \pi(g_1^{-1} \pi(g_1 g_2)))
$$
By Lemma \ref{le:impsi}, $ \pi(g_1g_2) = n g_1 g_2 n_- $ with $ n \in U^i[t], n_- \in U_-[t] $.  Thus
$$
f(m(g_1, g_2)) =  (g_1, \pi(g_1^{-1} n g_1 g_2 n_-)) = (g_1, \pi(g_2)) = (g_1, g_2)
$$
since $ g_1^{-1} n g_1 \in U^i[t] $, by Proposition \ref{pr:commuteUi}(3).
\end{proof}

Recall that we have a Hamiltonian $\Ga$ action $ \oW^0_{-\alpha_i^\vee}$ by Lemma \ref{lemma: hamiltonian action on W0alphai}, with moment map $\Phi_i$, and that $ \Ga $ acts freely and transitively on $ \Phi_i^{-1}(1)_{-\alpha_i^\vee}^0 \subset \oW^0_{-\alpha_i^\vee} $.  

Now we consider the $\mathbb{G}_a$-action on $\overline{\mathcal{W}}_{-\alpha_i^\vee}^0 \times \mathcal{W}_{\mu +\alpha_i^\vee}$ acting solely on the first component by $a\cdot (g_1, g_2) := (a\cdot g_1,g_2) = \big(\pi(x_{-i}(-d_i^{1/2}a)g_1), g_2\big)$.

\begin{Proposition}\label{mequivariant}
$m :\overline{\mathcal{W}}_{-\alpha_i^\vee}^0 \times \mathcal{W}_{\mu +\alpha_i^\vee} \rightarrow \Phi_i^{-1}(\Cx)_\mu $ is $\mathbb{G}_a$-equivariant.
\end{Proposition}
\begin{proof}
Let $(g_1,g_2)\in \overline{\mathcal{W}}_{-\alpha_i^\vee}^0 \times \mathcal{W}_{\mu +\alpha_i^\vee}$. For $a\in \mathbb{G}_a$, we have to show that
\begin{align*}
\pi\big(x_{-i}(-d_i^{1/2}a)g_1g_2\big)=\pi\big(x_{-i}(-d_i^{1/2}a)\pi(g_1g_2)\big).
\end{align*}
By Lemma \ref{le:g1g2}(3), there exists $ n \in U^i[t], n_- \in U_-[t] $ such that $ \pi(g_1 g_2) = n g_1 g_2 n_-$.  By Proposition \ref{pr:commuteUi}(3), $ x_{-i}(-d_i^{1/2}a) n x_{-i}(-d_i^{1/2}a)^{-1} \in U^i[t] $, so we see that,
$$
\pi\big(x_{-i}(-d_i^{1/2}a) ng_1 g_2 n_-\big) = \pi\big(x_{-i}(-d_i^{1/2}a) n x_{-i}(-d_i^{1/2}a)^{-1} x_{-i}(-d_i^{1/2}a) g_1 g_2 n_-\big),
$$
which equals $\pi\big(x_{-i}(-d_i^{1/2}a) g_1 g_2 \big)$ as desired.

\end{proof}

Combining Theorem \ref{psiminverses}, Proposition \ref{mequivariant}, and Theorem \ref{multpoi}, we thus have a $ \Ga $--equivariant Poisson isomorphism $   \oW^0_{-\alpha_i^\vee} \times \cW_{\mu + \alpha_i^\vee} \rightarrow \Phi_i^{-1}(\Cx)_\mu$. From this isomorphism and Lemma \ref{le:g1g2} (3), we can deduce the following result.

\begin{Corollary}
The action of $ \Ga $ on $ \cW_\mu $ is Hamiltonian with moment map $\Phi_i $.
\end{Corollary}

\vskip1ex
Having the equivariant isomorphism $m: \overline{\mathcal{W}}_{-\alpha_i^\vee}^0 \times \mathcal{W}_{\mu +\alpha_i^\vee} \rightarrow  \Phi_i^{-1}(\mathbb{C}^\times)_\mu$, we also obtain an isomorphism for the generalized slices.

\begin{Corollary}
\label{cor: reduction for slices v1}
 Let $\lambda$ be a dominant coweight and let $ \mu \le \lambda $ be a coweight. Consider $\Phi_i^{-1}(\mathbb{C}^\times)^\lambda_\mu \subset \oW^\lambda_\mu $.  Then  $m$ restricts to a $\mathbb{G}_a$-equivariant isomorphism $$m: \overline{\mathcal{W}}_{-\alpha_i^\vee}^0 \times \oW^\lambda_{\mu+\alpha_i^\vee} \longrightarrow \Phi_i^{-1}(\mathbb{C}^\times)^\lambda_\mu$$
\end{Corollary}
\begin{proof}
This follows from the fact that the maps $m$ and $\pi$ preserve the subsets $ \overline{G[t]t^\lambda G[t]}$.
\end{proof}

We arrive at the desired reduction result:

\begin{Corollary}
\label{cor: reduction for slices v2}
Let $ \lambda $ be a dominant coweight and let $ \mu \le \lambda $ be a coweight.  The map $ g \mapsto \pi(\xi_i(g)^{-1}g) $ gives an isomorphism
$$
\oW^\lambda_\mu \sslash_1 \Ga \cong \oW^\lambda_{\mu + \alpha_i^\vee}
$$
\end{Corollary}

\subsection{The rank 1 case}
Take $ G = PGL_2 $.  Let $ \lambda \in \N $ be a dominant coweight and let $ \mu = \lambda - 2m $, for $ m \in \N $.  Then by \cite[Prop 2.17]{BFN}, we have
\begin{align*}
\oW^\lambda_\mu = \left\{ g = \begin{bmatrix} a & b \\ c & d \end{bmatrix} : \begin{array}{cl}(i) & a,b,c,d \in \C[t], \\ (ii) & \text{$d $ is monic of degree $m$}, \\ (iii) &\text{$b,c$ have degree smaller than $ m $}, \\ (iv) & \det(g) = t^\lambda \end{array}  \right\}
\end{align*}
In this case, let us write $ b = b^{(1)} t^{m-1} + \dots, \ d = t^m + d^{(1)}t^{m-1} + \dots $.  We have that $ \Phi(g) = b^{(1)} $.  Also, we see that
$$
\xi(g) = \begin{bmatrix} 0 & b^{(1)} \\ -{b^{(1)}}^{-1} & t - {b^{(1)}}^{-1} b^{(2)} + d^{(1)} \end{bmatrix}
$$

Thus, the map
$\Phi^{-1}(1) \rightarrow \oW^\lambda_{\mu + \alpha_i^\vee}
$
is given by
\begin{equation} \label{eq:d'b}
g \mapsto \pi(\xi(g)^{-1} g) = \begin{bmatrix} a' & b' \\ c' & d' \end{bmatrix} \text{ where $ b' = b(t- b^{(2)} + d^{(1)}) - d, \ d' = b $}
\end{equation}

Recall that since $ \oW^\lambda_\mu $ is a Coulomb branch, it comes with an integrable system (see \cite[(3.17)]{BFN1}). In the case of $ PGL_2 $ this integrable system takes the form 
$$ \oW^\lambda_\mu \rightarrow \C^m / S_m \quad \begin{bmatrix} a & b \\ c & d \end{bmatrix}  \mapsto d $$

We can consider the restriction of this integrable system to the locus $ \Phi^{-1}(1) $ and also consider the map $ \Phi^{-1}(1) \rightarrow \oW^\lambda_{\mu + \alpha} \rightarrow \C^{m-1} / S_{m-1} $.  Using (\ref{eq:d'b}), we see that the map $ \Phi^{-1}(1) \rightarrow \C^m / S_m \times \C^{m-1} / S_{m-1} $ is given by $ g \mapsto (d,b) $ which is birational (it is the same as the morphism to the Zastava space).

Thus the correspondence $ \Phi^{-1}(1) \rightarrow \oW^\lambda_\mu \times \oW^\lambda_{\mu + \alpha} $ is a birational section of the integrable system.  This is a special feature of the $ PGL_2 $ case.

\section{Quantum Hamiltonian reduction}
In this section, we quantize the isomorphism $\Phi_i^{-1}(\mathbb{C}^\times)_\mu \simeq \overline{\mathcal{W}}_{-\alpha_i^\vee}^0 \times \mathcal{W}_{\mu+\alpha_i^\vee}$.

Our first task is to  find a quantization of $\Phi_i^{-1}(\mathbb{C}^\times)_\mu$. Since $\Phi_i$ is the classical limit of $E_i^{(1)}$, our quantization will correspond to the localization of $Y_\mu$ at $E_i^{(1)}$.

\subsection{A localization for $Y_\mu$}\label{alocalizationforYmu}

Let $S$ be a multiplicative subset of a ring $A$, meaning that $S$ is closed under multiplication, $1\in S$, and $0 \notin S$.  	We call $S$ a \textbf{right denominator set} if:
\begin{enumerate}
\item $a S \cap s A \neq \emptyset$ for all $a \in A$ and $s\in S$ (the right Ore condition),
\item for any $a\in A$, if $sa = 0$ for some $s\in S$, then $a t =0$ for some $t\in S$.
\end{enumerate}
By Ore's Theorem \cite[Theorem 10.6]{La}, the right ring of fractions $A S^{-1}$ exists if and only if $S$ is a right denominator set.  One may similarly define the notion of left denominator set, which is equivalent to the existence of the left ring of fractions $S^{-1} A$.  By \cite[Corollary 10.14]{La}, if $ AS^{-1}$ and $S^{-1} A$ both exist then 
$$
A S^{-1} \cong S^{-1} A
$$

\cite[Prop.~1.3.8]{Ginzburg} gives  a simple criterion for proving that $S$ is a denominator set:

\begin{Lemma}
Let $S$ be a multiplicative subset of a ring $A$.  Suppose that for all $s\in S$, the operator 
$$
\operatorname{ad}_s: A \rightarrow A, \qquad a\mapsto [s,a]
$$
is locally nilpotent.  Then $S$ is both a right and left denominator set.  In particular, the rings of fractions $A S^{-1} \cong S^{-1} A$ both exist.
\end{Lemma}	

The next lemma follows by induction, using the Leibniz rule:

\begin{Lemma}
Let $r\in A$, such that $\ad_r$ is nilpotent.  Then $\ad_{r^n}$ is also nilpotent, for any $n\geq 0$.
\end{Lemma}

Combining the above lemmas, we get an even simpler criterion:
\begin{Corollary} \label{co:GabberSimple}
Let $s\in A$ be an element which is not nilpotent, and let $S = \left\{s^n : n \geq0\right\}$.  If $\ad_s$ is locally nilpotent on $A$, then $S$ is both a right and left denominator set, and the rings of fractions $A S^{-1} \cong S^{-1} A$ both exist.
\end{Corollary}

\begin{Theorem}
\label{OreYmu}
Let $S = \left\{ (E_i^{(1)})^n : n \geq 0\right\} \subseteq Y_\mu$.  Then $S$ is both a right and left denominator set, and the rings of fractions $Y_\mu S^{-1} \cong S^{-1} Y_\mu$ both exist.  
\end{Theorem}
We will often write 
\begin{equation*}
Y_\mu[ (E_i^{(1)})^{-1}] = Y_\mu S^{-1} \cong S^{-1} Y_\mu
\end{equation*}

\begin{proof}
Since $Y_\mu$ is a domain,  $E_i^{(1)}$ is not nilpotent.  Thus $S$ is a multiplicative set.

By the above corollary, it suffices to prove that the adjoint action of $E_i^{(1)}$ is locally nilpotent.  We can reduce to the case where $\mu$ is antidominant and satisfies $\langle \mu, \alpha_i \rangle < -1 $.  Indeed, the embedding $\iota_{\mu, 0, \mu_2}: Y_\mu \hookrightarrow Y_{\mu+\mu_2}$  is equivariant for $[E_i^{(1)}, -]$, so local nilpotency on $Y_{\mu+\mu_2}$ implies local nilpotency on $Y_\mu$.  Choosing some sufficiently antidominant $\mu_2$, we can thus replace $\mu$ by $\mu+\mu_2$.

By the Leibniz rule, it further suffices  to prove that $[{E_i^{(1)}}, -]$ acts nilpotently on generators of $Y_\mu$.  We can thus reduce to checking that $[E_i^{(1)}, -]$ acts nilpotently on the generators from Lemma \ref{4generators}.  For $E_j^{(1)}$ this follows from the Serre relations, while 
$$[E_i^{(1)}, [E_i^{(1)}, S_j^{(-\langle \mu, \alpha_j\rangle + 1)}] ]= [E_i^{(1)}, -( \alpha_j \cdot \alpha_i ) E_i^{(1)}] = 0$$
Next, recall that 
$$
[ E_i^{(1)}, S_j^{(-\langle \mu, \alpha_j\rangle + 2)} ] = -( \alpha_j \cdot \alpha_i ) E_i^{(2)}
$$
The case $r=1, s=1$ of relation (\ref{E,E}) tells us that
$$
[E_i^{(1)}, E_i^{(2)}] = - \frac{\alpha_i\cdot \alpha_i}{2} ( E_i^{(1)})^2,
$$
so together we see that $[E_i^{(1)}, [E_i^{(1)}, [E_i^{(1)}, S_j^{(-\langle \mu, \alpha_j\rangle + 2)}] ] ]  = 0$.  Finally, we know that $[E_i^{(1)}, F_j^{(1)}] = \delta_{ij} H_i^{(1)}= 0 $ since we took $\langle \mu, \alpha_i \rangle < -1 $.

\end{proof}

Recall from section \ref{se:filtrations} that given any splitting $\mu=\nu_1+\nu_2$, we have a filtration $F_{\nu_1,\nu_2} Y_\mu$. Now, following \cite[12.3]{S}, we can put a filtration on $Y_\mu[(E_i^{(1)})^{-1}]$ as follows. Since $Y_\mu$ is a domain, given $x\in Y_\mu, s\in S=\{(E_i^{(1)})^n: n\in \mathbb{N}\}$,  we specify the degree $\deg(xs^{-1})=\deg(x)-\deg(s)$.

\begin{Proposition}
$\gr Y_\mu[(E_i^{(1)})^{-1}]\simeq \mathbb{C}[\Phi_i^{-1}(\mathbb{C}^\times)]$.
\end{Proposition}
\begin{proof}
This is a special case of a general statement on localization of filtered rings (see \cite[II,3.2]{LR}, \cite[Prop.~12.5]{S}).
\end{proof}

Recall from Lemma \ref{lemma: simplest coulomb branch of quiver gauge theory} that the algebra $Y_{-\alpha_i^\vee}^0$ is generated by elements $A_i^{(1)},(E_i^{(1)})^{\pm 1}$ with the relation $[E_i^{(1)},A_i^{(1)}]=d_i E_i^{(1)}$.

\begin{Proposition} \label{pr:tildeDeltaExists}
There exists a map $\tilde{\Delta}: Y_\mu[(E_i^{(1)})^{-1}] \longrightarrow Y_{-\alpha_i^\vee}^0 \otimes Y_{\mu+\alpha_i^\vee}$.
\end{Proposition}
\begin{proof}
Consider $\overline{\Delta}:Y_\mu \xrightarrow{\Delta} Y_{-\alpha_i^\vee} \otimes Y_{\mu+\alpha_i^\vee} \rightarrow Y_{-\alpha_i^\vee}^0 \otimes Y_{\mu+\alpha_i^\vee}$. We see that $\overline{\Delta}(E_i^{(1)})=E_i^{(1)}\otimes 1$ (as $ \langle \alpha^\vee_i, \alpha_i \rangle = 2) $. Since $E_i^{(1)}$ is invertible in $Y_{-\alpha_i^\vee}^0$, $\tilde{\Delta}$ exists by the universal property of Ore localization \cite[Corollary 10.11]{La}.
\end{proof}

\subsection{Lifting the isomorphism}\label{liftingiso}
In order to prove that $ \tilde{\Delta} $ is an isomorphism, we consider filtrations of $Y_\mu$.

Recall that, for coweights $\nu_1,\nu_2$ such that $\mu=\nu_1+\nu_2$, we have the filtration $F_{\nu_1,\nu_2}Y_\mu$ (section \ref{se:filtrations}).

\begin{Lemma}\label{desiredcompatibility}
Consider the filtrations $F_{\nu,\mu-\nu}Y_\mu, F_{\nu,-\alpha_i^\vee-\nu}Y_{-\alpha_i^\vee}^0, F_{\alpha_i^\vee+\nu,\mu-\nu}Y_{\mu+\alpha_i^\vee}$. Then $\tilde{\Delta}: Y_\mu[(E_i^{(1)})^{-1}] \rightarrow Y_{-\alpha_i^\vee}^0\otimes Y_{\mu+\alpha_i^\vee}$ respects these filtrations.
\end{Lemma}
\begin{proof}
This follows from Proposition \ref{comultfiltrcomp}.
\end{proof}

We would like to use the following lemma.
\begin{Lemma}\label{liftgradediso}
Let $\phi: A\rightarrow B$ be a map of $\mathbb{Z}$-filtered vector spaces with increasing filtrations. Assume that all involved filtrations are exhaustive. Additionally, assume that the filtration on $A$ is separated, i.e., $\bigcap_n A_n= 0$. Denote by $\gr \phi:\gr A \rightarrow \gr B$ the induced map on the associated graded level.
\begin{itemize}
\item[(1)] If $\gr \phi$ is injective, so is $\phi$.
\item[(2)] Suppose that $B_n=0$ for all $n<0$. If $\gr \phi$ is surjective, so is $\phi$.
\end{itemize}
\end{Lemma}
\begin{proof}
(1) Assume that $\gr\phi$ is injective. Let $ a \in A $ and suppose that $\phi(a)=0$. Assume that $a \neq 0$. Since the filtration is separated and exhaustive, there exists $d$ such that $a\in A_d$, and $a\not \in A_{d-1}$. For $\bar{a}\in A_d/A_{d-1}$, since $\phi(a)=0$, $\gr \phi( \bar{a}) = \overline{\phi(a)} =0$. Since $\gr \phi$ is injective, $\bar{a} = 0$. This means that $a \in A_{d-1}$, a contradiction. Hence, $a=0$.

\vskip1ex
(2) Assume that $B_n=0$ for all $n<0$. We prove by induction on $d$ that $\phi: A_d\rightarrow B_d$ is surjective.
Suppose that $b\in B_0$. Since $\gr \phi$ is surjective, there exists $a \in A_0$ such that $\overline{\phi(a)}=\bar{b}$. Since $ B_{-1} = 0 $, this implies that $ \phi(a) = b $.

Suppose that $\phi : A_d \rightarrow B_d $ is surjective. Suppose that $b\in B_{d+1}$. Since $ \gr \phi $ is surjective, there exists $ a \in A_{d+1} $ such that $ \overline{\phi(a)} = \bar{b} $ and hence $ b - \phi(a) \in B_d $.  By the induction hypothesis, there exists $a'\in A_d$ such that $b - \phi(a) =\phi(a')$. Therefore, $b=\phi(a+a')$.
\end{proof}

 In general the filtrations on $Y_\mu$ are not bounded below.  So in order to establish the surjectivity of $ \tilde \Delta $, it will be important to find such filtrations.

\begin{Lemma}\label{nonnegativecondition}
Let $ \mu, \nu $ be two coweights.  Suppose that
\begin{itemize}
\item[(i)] $\langle \nu ,\alpha_i\rangle =-1$,
\item[(ii)] for all positive roots $\beta$, $\langle \nu, \beta \rangle \geq -1$,
\item[(iii)] for all positive roots $\beta$, $\langle \mu -\nu,\beta \rangle \geq -1$,
\end{itemize}
Then $E_i^{(1)} $ has filtered degree zero in $ Y_\mu $, and the filtrations $ F_{\nu,-\alpha_i^\vee-\nu}Y_{-\alpha_i^\vee}^0 $ and $F_{\alpha_i^\vee+\nu,\mu-\nu}Y_{\mu+\alpha_i^\vee}$ are non-negative. 
\end{Lemma}
\begin{proof}
The degree of $ E_i^{(1)} $ is clear, as is the filtration on $ Y_{-\alpha_i^\vee}^0 $.  So we consider the filtration $F_{\alpha_i^\vee+\nu,\mu-\nu}Y_{\mu+\alpha_i^\vee}$.
	
The degrees of $ H_j^{(r)}$ are always positive, so we inspect the degrees of $ E_\beta^{(r)},F_\beta^{(r)}$ in $F_{\alpha_i^\vee + \nu,\mu-\nu}Y_{\mu+\alpha_i^\vee}$,
\begin{align*}
\deg (F_\beta^{(r)})&= \langle \mu-\nu ,\beta \rangle + r \geq r-1 \geq 0, \\
\deg (E_\beta^{(r)})&= \langle \alpha_i^\vee+\nu ,\beta \rangle + r = \langle s_i(\nu) ,\beta \rangle +r = \langle \nu,s_i(\beta) \rangle +r.
\end{align*}
If $\beta=\alpha_i$, then $\langle \nu,s_i(\beta)\rangle =1$. If $\beta\neq \alpha_i$, then $s_i(\beta)$ is a positive root not equal to $\alpha_i$, and so $\langle \nu,s_i(\beta)\rangle \geq -1$. Thus, $\deg(E_\beta^{(r)})\geq 0$.
\end{proof}

\begin{Theorem} \label{th:domiso}
For sufficiently dominant $ \mu $, the map $\tilde{\Delta}: Y_\mu[(E_i^{(1)})^{-1}] \longrightarrow Y_{-\alpha_i^\vee}^0\otimes Y_{\mu+\alpha_i^\vee}$ is an isomorphism.
\end{Theorem}

\begin{proof}
Let $ r $ be the coefficient of $ \alpha_i $ in the longest root.  Let $ \nu = - \varpi^\vee_i + r \sum_{j \ne i} \varpi^\vee_j $.  Then $ \langle \nu, \alpha_i \rangle = - 1 $.  For any other positive root $ \beta $, the coefficient of $\alpha_j $ is non-zero, for some $ j \ne i$, and the coefficient of $ \alpha_i $ is at most $ r $.  Thus $ \langle \nu, \beta \rangle \ge -1 $.

Choose any dominant $ \mu $ such that $ \mu - \nu $ is dominant.  Then the conditions of Lemma \ref{nonnegativecondition} are satisfied and so the filtration on $ Y_{-\alpha_i^\vee}^0 \otimes Y_{\mu + \alpha_i^\vee} $ is non-negative.  Thus combining Theorem \ref{psiminverses} and Lemma \ref{liftgradediso}, the result follows.
\end{proof}

Now, we can push the argument further with the next few lemmas. Denote by $Y_\mu^<, Y_\mu^\leq$ the subalgebras of $Y_\mu$ generated by the $F_j^{(r)}$ (resp. $F_j^{(r)}$ and $H_j^{(s)}$). Similarly, denote by $Y_\mu^>, Y_\mu^\geq$ the subalgebras generated by the $E_j^{(r)}$ (resp. $E_j^{(r)}$ and $H_j^{(s)}$).

\begin{Lemma}\label{egammafj}
For all positive roots $\gamma$, $j\in I$ and all  $r,q>0$, we have
\begin{gather*} [E_\gamma^{(q)},F_j^{(r)}] \in Y_\mu^\geq,\ [F_\gamma^{(q)},E_j^{(r)}] \in Y_\mu^\leq, \\ 
[S_j^{(-\langle \mu,\alpha_j\rangle +2)},E_\gamma^{(q)}] \in Y_\mu^>\text{, and } [S_j^{(-\langle \mu,\alpha_j\rangle +2)},F_\gamma^{(q)}] \in Y_\mu^<
\end{gather*}
\end{Lemma}
\begin{proof}
We shall only prove $[E_\gamma^{(q)},F_j^{(r)}]\in Y_\mu^\geq$ since the other statements are similar. For $l\geq 1$, consider a commutator $x=[E_{i_1}^{(q_1)},[E_{i_2}^{(q_2)}, \cdots[E_{i_{l-1}}^{(q_{l-1})},E_{i_l}^{(q_l)} ]\cdots ]$ where $q_1,\dots,q_l$ are arbitrary positive integers. We show that $[x,F_j^{(r)}]\in Y_\mu^\geq$.

We proceed by induction on $l$. If $l=1$, then $[E_{i_1}^{(q_1)},F_j^{(r)}] = \delta_{ji_1}H_j^{(q_1+r-1)}$. For the induction step $l+1$, let $y=[E_{i_2}^{(q_2)}, \cdots[E_{i_{l}}^{(q_{l})},E_{i_l}^{(q_{l+1})} ]\cdots ]$.
\begin{align*}
[x,F_j^{(r)}] &= [[E_{i_1}^{(q_1)}, y],F_j^{(r)}] = [E_{i_1}^{(q_1)},[y,F_j^{(r)}]] - [y,\delta_{ji_1}H_j^{(r+q_1-1)}] \in Y_\mu^\geq
\end{align*}
since $[y,F_j^{(r)}]\in Y_\mu^\geq$ by induction hypothesis.
\end{proof}

\begin{Lemma}\label{lem:shiftrespectsborel}
Let $\mu$ be a coweights and let $ \mu_1, \mu_2 $ be antidominant coweights.  Let $\mu' = \mu + \mu_1 + \mu_2$. Consider the shift morphism $\iota=\iota_{\mu,\mu_1,\mu_2}: Y_\mu \longrightarrow Y_{\mu + \mu_1 + \mu_2}$. Then $\iota^{-1}(Y_{\mu'}^\geq) \subseteq Y_\mu^\geq$,  and $\iota^{-1}(Y_{\mu'}^<) \subseteq Y_\mu^<$.
\end{Lemma}
\begin{proof}
We prove the first equality, the proof for the second is similar. 

We use an idea from the proof of \cite[Cor.~3.15]{FKPRW}. Given a choice of PBW variables $E_\beta^{(r)}, F_\beta^{(r)}$ in $Y_\mu$, we consider their images under $\iota$ as PBW variables in $Y_{\mu'}$, and extend them to a full set of PBW variables in $Y_{\mu'}$. In particular, under these conditions, the PBW basis of $Y_\mu$ maps bijectively to a subset of the corresponding PBW basis of $Y_{\mu'}$. 

Note that $Y_{\mu'}^\geq $ is spanned by a subset of the PBW basis as well (those monomials which contain no $F$s).  Since $\iota $ is injective by \cite[Cor.~3.16]{FKPRW}, in computing $\iota^{-1}(Y_{\mu'}^\geq) $,  we are simply intersecting two subspaces of $ Y_{\mu'} $, each of which are spanned by a subset of our PBW basis.  Thus, we can simply intersect the corresponding sets of basis vectors which are all contained in $ Y_\mu^\geq $.  This gives the  desired result.

\end{proof}

\begin{Lemma}\label{Ygeq}
Let $ \mu_1, \mu_2 $ be any two coweights.  Consider $\Delta_{\mu_1,\mu_2}: Y_{\mu_1+\mu_2}\longrightarrow Y_{\mu_1}\otimes Y_{\mu_2}$. For $r\geq 1$ and $j\in I$,
\begin{align*}
\Delta(F_j^{(r)}) &\in 1\otimes F_j^{(r)} + Y_{\mu_1}^<\otimes Y_{\mu_2}^\geq 
\end{align*}
\end{Lemma}
\begin{proof}
First, consider the case where $\mu_1,\mu_2$ are antidominant. We proceed by induction on $ r$. First note that $$\Delta(F_j^{(1)})=1\otimes F_j^{(1)} + \delta_{0,\langle \mu_2,\alpha_j\rangle}F_j^{(1)}\otimes 1 \in 1\otimes F_j^{(1)} + Y_{\mu_1}^<\otimes Y_{\mu_2}^\geq$$ Recall that
\[
\Delta(S_j^{(-\langle \mu,\alpha_j\rangle +2 )})= S_j^{(-\langle \mu_1,\alpha_j\rangle +2 )}\otimes 1 + 1\otimes S_j^{(-\langle \mu_2,\alpha_j\rangle +2 )} - \sum\limits_{\gamma>0}\langle \alpha_j ,\gamma \rangle F_\gamma^{(1)} \otimes E_\gamma^{(1)}.
\]
 Since $[S_j^{(-\langle \mu,\alpha_j\rangle +2 )},F_j^{(r)}]=-2F_j^{(r+1)}$ and since $\Delta$ is an algebra homomorphism, the induction step follows from Lemma \ref{egammafj}.

For the general case, choose antidominant $\eta_1,\eta_2$ such that $\mu_1+\eta_1$, $\mu_2+\eta_2$ are antidominant. Consider the commutative diagram from Theorem \ref{coproduct}
\[
\xymatrix{
Y_{\mu_1+\mu_2} \ar[rr] \ar[d]_{\iota_{\mu_1+\mu_2, 0,\eta}} & & Y_{\mu_1} \otimes Y_{\mu_2} \ar[d]^{\iota_{\mu_1,\eta_1,0}\otimes \iota_{\mu_2,0,\eta_2}} \\
Y_{\mu_1+\mu_2 +\eta_1+\eta_2} \ar[rr] & & Y_{\mu_1+\eta_1} \otimes Y_{\mu_2+\eta_2}
}
\]
The result follows from the commutativity of the diagram and Lemma \ref{lem:shiftrespectsborel}. 
\end{proof}

\begin{Theorem}\label{tildeDeltaiso}
For any coweight $ \mu $, the map
$\tilde{\Delta}: Y_\mu[(E_i^{(1)})^{-1}] \longrightarrow Y_{-\alpha_i^\vee}^0\otimes Y_{\mu+\alpha_i^\vee}$ is an isomorphism.
\end{Theorem}
\begin{proof}
Combining Theorem \ref{psiminverses} and Lemma \ref{liftgradediso}, we see that $ \tilde{\Delta} $ is injective.
 
In order to show surjectivity, let us choose dominant $\eta$ such that $\mu+\eta$ is sufficiently dominant. By Theorem \ref{th:domiso},
$Y_{\mu+ \eta}[(E_i^{(1)})^{-1}] \rightarrow Y_{-\alpha_i^\vee}^0\otimes Y_{\mu+ \eta+ \alpha_i^\vee}$ is an isomorphism.

 Consider the commutative diagram
\[
\xymatrix{
Y_{\mu+\eta}[(E_i^{(1)})^{-1}] \ar[rr]^\sim \ar[d]_{\iota_{\mu+\eta, 0,-\eta}} & & Y_{-\alpha_i^\vee}^0 \otimes Y_{\mu+\eta+\alpha_i^\vee} \ar[d]^{\Id \otimes \iota_{\mu+\eta+\alpha_i^\vee,0,-\eta}} \\
Y_\mu[(E_i^{(1)})^{-1}] \ar[rr] &  & Y_{-\alpha_i^\vee}^0 \otimes Y_{\mu+\alpha_i^\vee}
}
\]
We have shown already that the top arrow is an isomorphism, since $\mu+\eta$ is dominant. By tracing around the diagram, we see that the image of the bottom arrow contains the image of the map $\Id \otimes \iota_{\mu+\eta +\alpha_i^\vee, 0,-\eta}$. Since $\iota_{\mu+\eta+\alpha_i^\vee,0,-\eta}$ restricts to an isomorphism $Y_{\mu+\eta+\alpha_i^\vee}^{\geq } \stackrel{\sim}{\rightarrow} Y_{\mu+\alpha_i^\vee}^{\geq}$, we see that the image of $\tilde{\Delta}: Y_\mu[(E_i^{(1)})^{-1}] \longrightarrow Y_{-\alpha_i^\vee}^0 \otimes Y_{\mu+\alpha_i^\vee}$ contains the subalgebra $Y_{-\alpha_i^\vee}^0 \otimes Y_{\mu + \alpha_i^\vee}^\geq $.

By Lemma \ref{Ygeq}, for all $r\geq 1$, $\Delta(F_j^{(r)})=1\otimes F_j^{(r)} + Y_{-\alpha_i^\vee}^<\otimes Y_{\mu +\alpha_i^\vee}^\geq$. Thus $ 1 \otimes F_j^{(r)} $ lies in the image of $\tilde{\Delta} $.  Hence $ \tilde{\Delta} $ is surjective and thus an isomorphism.
\end{proof}

\subsection{Quantum Hamiltonian reduction}
Let $A$ be an associative algebra, and let $x \in A$.  Assume that $\ad_x $ is locally nilpotent.  Then we can define an action of $ \Ga $ on $ A$ where the action of $ a \in \Ga $ is given by $ \exp(a\ad_x) $.  Define a linear map $\mathbb{C} \rightarrow A$, $z \mapsto zx$; this is the quantum moment map for this action.  

Recall that the \emph{quantum Hamiltonian reduction} of $A$ by $\mathbb{G}_a$ is the algebra defined by:
\begin{equation}
A \sslash_1 \mathbb{G}_a = \left\{ a \in A / A(x-1) \ : \ (x-1)a \in A (x-1) \right\} \cong \End_A(A / A(x-1))^{op}
\end{equation}
See \cite{GanGinzburg} for generalities on quantum Hamiltonian reduction. Note that the condition $(x-1)a \in A(x-1)$ may be  written equivalently as $[x,a] \in A(x-1)$.

In our situation, we have a $ \Ga $-action on $ Y_\mu $, determined by the element $ E_i^{(1)}$.  We deduce the following consequence of Theorem \ref{tildeDeltaiso}.
\begin{Corollary}
\label{cor: quantum hamiltonian reduction for shifted Yangians}
There is an isomorphism $Y_{\mu+\alpha_i^\vee} \cong Y_\mu \sslash_1 \mathbb{G}_a$.
\end{Corollary}

\begin{proof}
	Let us return to the general situation of $ A, x $ above.  Assume moreover, that $ x $ is not nilpotent.  By Corollary \ref{co:GabberSimple}, we see that $ A[x^{-1}] $ exists and it is easy to see that the natural map $ A / A(x-1) \rightarrow A[x^{-1}] / A[x^{-1}](x-1) $ gives an isomorphism $ A \sslash_1 \Ga \rightarrow A[x^{-1}] \sslash_1 \Ga $.

	Thus, by Theorem \ref{tildeDeltaiso}, it suffices to study the quantum Hamiltonian reduction of $ A =  Y_{-\alpha_i^\vee}^0 \otimes Y_\mu $ by the element $ x = \tilde \Delta(E_i^{(1)}) = E_i^{(1)} \otimes 1 $ (see the proof of Proposition \ref{pr:tildeDeltaExists}).  
The proof now follows from the next lemma.
\end{proof}

\begin{Lemma}
There is an isomorphism $Y_{-\alpha_i^\vee}^0 \sslash_1 \Ga \cong \C$.
\end{Lemma}
\begin{proof}
Recall from Lemma \ref{rem: cotangent bundle of C*} that $Y_{-\alpha_i^\vee}^0 \cong D(\Cx)$ with $E_i^{(1)} \mapsto z$. The algebra $D(\Cx)$ has $\C$-basis $\{ (\frac{\partial}{\partial z})^m z^n : m,n\in\Z, m\geq 0\}$, so $D(\Cx) / D(\Cx) (z - 1) $ has basis $\{ (\frac{\partial}{\partial z})^m : m\geq 0\}$. By induction on $m$, or by using the Fourier transform, one can show that 
$$
[z, (\tfrac{\partial}{\partial z})^m] = - m (\tfrac{\partial}{\partial z})^{m-1}
$$
For $ m \ne 0 $, these expressions are linearly independent in $D(\Cx) / D(\Cx) (z - 1) $. Therefore $Y_{-\alpha_i^\vee}^0 \sslash_1 \Ga \cong D(\Cx) \sslash_1 \Ga \cong \C$.
\end{proof}

%%%%%%
\section{Towards quantum Hamiltonian reduction for generalized slices}

\subsection{A conjecture and a weak form}
 Let $\lambda$ be a dominant coweight and let $ \mu \le \lambda $ be a coweight such that $ \mu + \alpha_i^\vee \le \lambda $.

 Considering $\Phi_i^{-1}(\C^\times)^\lambda_\mu \subset \overline{\cW}^\lambda_\mu$, we have seen in Corollary \ref{cor: reduction for slices v1} that the multiplication map $m$ induces an isomorphism
 $$
 m:  \overline{\mathcal{W}}_{-\alpha}^0 \times \oW^\lambda_{\mu+\alpha_i^\vee}
  \stackrel{\sim}{\longrightarrow}  \Phi_i^{-1}(\mathbb{C}^\times)^\lambda_\mu
 $$
We turn now to the question of quantizing this isomorphism.

Recall from section \ref{section: tsy}, that in order to quantize $\overline{\cW}_\mu^\lambda$, we consider quotients $Y_\mu \twoheadrightarrow Y^\lambda_\mu(\bR)$ called truncated shifted Yangians.  We will now try to show that the isomorphism from Theorem \ref{tildeDeltaiso} is compatible with taking these quotients.

\begin{Rem}
The assumption that $ \mu + \alpha_i^\vee \le \lambda $ gives us that $ m_i > 0 $ (in the notation of section \ref{section: tsy} ) and then the formulas from Theorem \ref{GKLO homomorphism} imply that $ E_i^{(1)} $ is non-zero in $ Y^\lambda_\mu(\bR)$.

Note also that $Y^\lambda_\mu(\bR)$ is a domain, so Theorem \ref{OreYmu} implies that the image of $S = \left\{(E_i^{(1)})^k: k\geq 0\right\}$ in $Y^\lambda_\mu(\bR)$ is still both a right and left denominator set.
\end{Rem}

\begin{Conjecture}
\label{truncated tildeDeltaiso}
For any dominant coweight $\lambda$ and coweight $\mu \leq \lambda$, such that $ \mu + \alpha_i^\vee \le \lambda $, there exists an isomorphism $ Y^\lambda_\mu(\bR)[(E_i^{(1)})^{-1}] \xrightarrow{\sim}  Y_{-\alpha_i^\vee}^0\otimes Y^\lambda_{\mu+\alpha_i^\vee} (\bR)$ fitting into a commutative diagram
$$
\xymatrix{
Y_\mu[(E_i^{(1)})^{-1}] \ar[r]^-\sim\ar@{->>}[d] & Y_{-\alpha_i^\vee}^0\otimes Y_{\mu+\alpha_i^\vee} \ar@{->>}[d] \\
Y^\lambda_\mu(\bR)[(E_i^{(1)})^{-1}] \ar[r]^-\sim &  Y_{-\alpha_i^\vee}^0\otimes Y^\lambda_{\mu+\alpha_i^\vee} (\bR)
}
$$
where the top horizontal arrow is from Theorem \ref{tildeDeltaiso}, while the the vertical arrows are the canonical quotient maps.
\end{Conjecture}
Assuming this conjecture, we can describe one truncated shifted Yangian as a {quantum} Hamiltonian reduction of another, similarly to Corollary \ref{cor: quantum hamiltonian reduction for shifted Yangians}:
\begin{Corollary}
\label{Cor: quantum Hamiltonian reduction of truncated shifted Yangians}
There is an isomorphism $Y_{\mu+\alpha_i^\vee}^\lambda(\bR) \cong Y_\mu^\lambda(\bR) \sslash_1 \mathbb{G}_a $, determined by the element $E_i^{(1)} \in Y_\mu^\lambda(\bR)$.
\end{Corollary}

As a first step toward Conjecture \ref{truncated tildeDeltaiso}, we have the following Lemma.

\begin{Lemma} \label{le:don't need iso}
	In the setting of Conjecture \ref{truncated tildeDeltaiso}, if a map $Y^\lambda_\mu(\bR)[(E_i^{(1)})^{-1}] \rightarrow  Y_{-\alpha_i^\vee}^0\otimes Y^\lambda_{\mu+\alpha_i^\vee} (\bR)$ exists which makes the diagram commutative, then it is an isomorphism.
\end{Lemma}

\begin{proof}
	It suffices to show that this arrow is surjective and injective.  Surjectivity is immediate from the commutativity of the diagram and the surjectivity of the top horizontal and right vertical arrows.
	
	For injectivity, we know that upon applying associated graded, we obtain the diagram
	$$
\xymatrix{
	\C[\Phi^{-1}(\Cx)_\mu] \ar[r]^-\sim\ar@{->>}[d] & \C[\cW^0_{-\alpha_i^\vee}] \otimes \C[\cW_{\mu+\alpha_i^\vee}] \ar@{->>}[d] \\
	\C[\Phi^{-1}(\Cx)_\mu^\lambda] \ar[r]^-\sim &  \C[\cW^0_{-\alpha_i^\vee}] \otimes \C[\cW_{\mu+\alpha_i^\vee}^\lambda]
}
$$ 
The bottom horizontal arrow here is unique given the rest of the diagram and thus must be the isomorphism from Corollary \ref{cor: reduction for slices v1}.  Thus by Lemma \ref{liftgradediso}.(1), the original bottom horizontal arrow is injective.
\end{proof}

At present, we do not know of a proof of Conjecture \ref{truncated tildeDeltaiso} in general, due to the fact that the defining ideal of $Y_\mu \twoheadrightarrow Y_\mu^\lambda(\bR)$ is difficult to characterize.  Recall from section \ref{section: tsy} that there is a conjectural presentation 
$$
\conjectural_\mu^\lambda(\bR) = Y_\mu /\langle A_j^{(r)} : r > m_j\rangle \twoheadrightarrow Y_\mu^\lambda(\bR),
$$  
i.e.~we conjecture this surjection to be an isomorphism.  In this section, we use this conjectural presentation to establish a weak version of Conjecture \ref{truncated tildeDeltaiso}.  

\begin{Theorem}
	\label{th: weak form of conjecture}
	The composed map
	\[
	\xymatrix{
		Y_\mu \ar[r]^<<<<\Delta &  Y_{-\alpha_i^\vee} \otimes Y_{\mu+\alpha_i^\vee}  \ar@{->>}[r] & Y_{-\alpha_i^\vee}^0 \otimes Y_{\mu+\alpha_i^\vee}^\lambda(\bR)
	}
	\]
	factors through $\conjectural_\mu^\lambda(\bR)$.  In particular, there is a commutative diagram
	\[
	\xymatrix{
		Y_\mu[(E_i^{(1)})^{-1}]  \ar@{->>}[d] \ar^-\sim[r] & Y_{- \alpha_i^\vee} \otimes Y_{\mu + \alpha_i^\vee} \ar@{->>}[d] \\
		\conjectural_\mu^\lambda(\bR)[(E_i^{(1)})^{-1}] \ar[r] & Y_{-\alpha_i^\vee}^0 \otimes Y_{\mu+\alpha_i^\vee}^\lambda(\bR) 
	}
	\]
\end{Theorem}
We will give the proof of this result in Section \ref{section: explicit formulas for Cartan generators} below.

Combining this Theorem with Lemma \ref{le:don't need iso}, we see that if $\conjectural_\mu^\lambda(\bR) = Y^\lambda_\mu(\bR) $, then Conjecture \ref{truncated tildeDeltaiso} holds.  Combining with Theorem \ref{appendix thm}, we deduce the following.

\begin{Corollary}
	Suppose that $\g$ is type A.  Then Conjecture \ref{truncated tildeDeltaiso} holds.
\end{Corollary}

\begin{Rem}
Suppose that $\g$ is type A and $\mu$ and $ \mu + \alpha_i^\vee$ are dominant.  Then by the quantized Mirkovi\'{c}-Vybornov isomorphism from \cite{QMV}, Corollary \ref{Cor: quantum Hamiltonian reduction of truncated shifted Yangians} translates to a Hamiltonian reduction between corresponding parabolic finite W-algebras of type A.  This is closely related to the works of Morgan \cite{Mor} and Rowe \cite{Ro}.
\end{Rem}

\subsection{The RTT presentation}
\label{section: RTT presentation}

We first recall some properties of the RTT presentation of the (unshifted) Yangian $Y = Y_0$.  This presentation was first explained by Drinfeld \cite[Theorem 6]{Drinfeld}, and a full proof appears in the recent work of Wendlandt \cite{Wen}.  

This construction takes as input a non-trivial finite-dimensional representation $V$ of the Yangian.  Following \cite[\S 2]{IR} we take $V = \bigoplus_i V(\varpi_i, 0)$, the sum of fundamental representations of the Yangian.  That is, $V(\varpi_i, 0)$ is the irreducible module for the Yangian which has Drinfeld polynomials $P_i(u) = u$ and $P_j(u) = 1$ for $j\neq i$. Having chosen $V$, there is an associated RTT presentation of the Yangian, given in \cite[Theorem 6.2]{Wen}.

For our purposes we will need only generators, not relations, and so the following formulation of this result will be sufficient: for each $r \geq 1$, there is a bilinear map 
\begin{equation}
V^\ast \times V \rightarrow Y, \qquad (\beta, w) \mapsto t_{\beta, w}^{(r)}
\end{equation}
defined as in \cite[\S 2.19]{IR}.  These elements generate $Y$, though we will not need this fact.  It is convenient to organize these elements into formal series $t_{\beta, w}(u) = \langle \beta, w\rangle + \sum_{r\geq 1} t_{\beta, w}^{(r)} u^{-r}$.

\begin{Rem}
Strictly speaking, the RTT presentation depends on a choice $\{e_a\}$ of basis for $V$.  Then $t_{e_a^\ast, e_b}^{(r)}$ corresponds to the RTT generator $t_{a,b}^{(r)}$ from \cite[\S 5.1]{Wen}, or more precisely its image in $Y$ under the isomorphism \cite[Theorem 6.2]{Wen}.
\end{Rem}

For us, the main advantage of this discussion is that the coproduct has a very natural description with respect to the above generators:

\begin{Lemma}
\label{lem: RTT coproduct}
Fix $i\in I$, and suppose that $\beta \in V(\varpi_i, 0)^\ast$ and $w \in V(\varpi_i, 0)$.  Then
$$
\Delta\big( t_{\beta, w}(u) \big) = \sum_v t_{\beta, v}(u) \otimes t_{v^\ast, w}(u),
$$
where the sum ranges over any fixed basis of $V(\varpi_i, 0)$.
\end{Lemma}
\begin{proof}
This is a reformulation of the formula for the coproduct given in \cite[\S 5.1]{Wen}, with one change: a priori the sum should run over a basis for all of $V = \bigoplus_j V(\varpi_j, 0)$.  But if $v \in V(\varpi_j,0)$ with $j\neq i$, then $t_{\beta, v}(u) = 0$ and so this larger sum is not needed. Indeed, this vanishing $t_{\beta,v}(u) = 0$ is immediate from the isomorphism \cite[Theorem 6.1]{Wen}, since each $V(\varpi_j, 0) \subset V$ is a subrepresentation for the Yangian.
\end{proof}

%%%
\subsubsection{GKLO generators}	
\label{section: gklo generators and rtt}

For any coweight $\mu$,  there are unique elements 
$$\gkloA_i^{(p)} \in \C[H_j^{(s)} : j \in I, s > - \langle\mu,\alpha_j\rangle] \subset Y_\mu$$ 
such that the following equation holds for all $i\in I$:
\begin{equation}
\label{eq: def of A gens without R}
H_i(u) =  u^{\langle \mu,\alpha_i\rangle}\frac{ \prod_{j \ne i} \prod_{r=1}^{-a_{ji}} \gkloA_j(u- \tfrac{1}{2} d_i a_{ij} - r d_j)}{ \gkloA_i(u) \gkloA_i(u-d_i)}
\end{equation}
Here $\gkloA_i(u) = 1 + \sum_{p \geq 1} \gkloA_i^{(p)} u^{-p}$. When $\mu=0$ these elements were introduced in \cite[Lemma 2.1]{GKLO}, and their existence and uniqueness for general $\mu$ follows by the same argument.
Define also
\begin{equation}
\label{eq: def of B and C without R}
\gkloB_i(u) = d_i^{1/2} \gkloA_i(u) E_i(u), \quad \gkloC_i(u) = d_i^{1/2} F_i(u) \gkloA_i(u)
\end{equation}

Although we have defined these elements for each $Y_\mu$ separately, they are in fact closely related to each other, via shift morphisms:
\begin{Lemma}
\label{lemma: GKLO ABC and shifts}
For any antidominant $\mu_1,\mu_2$ and any $p\geq 1$,
$$
\iota_{\mu,\mu_1,\mu_2}(\gkloA_i^{(p)}) = \gkloA_i^{(p)}, \quad \iota_{\mu,0,\mu_2}(\gkloB_i^{(p)}) = \gkloB_i^{(p)}, \quad \iota_{\mu,\mu_1,0}(\gkloC_i^{(p)}) = \gkloC_i^{(p)}
$$

\end{Lemma}
\begin{proof}
Definition \ref{shiftmaps} implies $\iota_{\mu,\mu_1,\mu_2}(H_i(u)) = u^{-\langle \mu_1+\mu_2,\alpha_i\rangle} H_i(u)$. Applying $\iota_{\mu,\mu_1,\mu_2}$ to equation (\ref{eq: def of A gens without R}) for $Y_\mu$ and multiplying by $u^{\langle \mu_1+\mu_2,\alpha_i\rangle}$, we see that the series $\iota_{\mu,\mu_1,\mu_2}(\gkloA_i(u))$ satisfy (\ref{eq: def of A gens without R}) for $Y_{\mu+\mu_1+\mu_2}$. By uniqueness, we must have $\iota_{\mu,\mu_1,\mu_2}(\gkloA_i(u)) = \gkloA_i(u)$.

Since $\iota_{\mu,0,\mu_2}(E_i(u)) = E_i(u)$ we see from (\ref{eq: def of B and C without R}) that $\iota_{\mu,0,\mu_2}(\gkloB_i(u)) = \gkloB_i(u)$. Similarly for $\gkloC_i(u)$.
\end{proof}

\begin{Lemma}
\label{lemma: gklo relations}
We have $\gkloB_i(u) = d_i^{1/2} [E_i^{(1)}, \gkloA_i(u)]$ and $\gkloC_i(u) = d_i^{1/2}[\gkloA_i(u), F_i^{(1)}]$.
\end{Lemma}
\begin{proof}
By applying shift morphisms and the previous lemma,  it suffices to prove the case when $\mu = 0$. This case follows from the relations \cite[Proposition 2.1]{GKLO}. 
\end{proof}

Recall that we have defined similar elements earlier in this paper: if $\lambda \geq \mu$ is a dominant coweight and $\bR$ is a set of parameters of coweight $\lambda$, then there are corresponding elements $A_i^{(r)} \in Y_\mu$ defined by equation (\ref{eq: def of A gens}).  

\begin{Lemma}
\label{lemma: various GKLO generators }
For any $\lambda\geq \mu$ and  $\bR$ be as above, there exist unique series $s_i(u) \in 1 + u^{-1} \C[[u^{-1}]]$ such that 
$$
A_i(u) = s_i(u) \gkloA_i(u)
$$
for all $i \in I$.
\end{Lemma}
\begin{proof}
The upper triangularity argument from the proof of \cite[Lemma 2.1]{GKLO} shows that there are unique series $s_i(u) \in 1 + u^{-1} \C[[u^{-1}]]$ such that
\begin{equation*}
 \frac{\prod_{j \ne i}\prod_{r=1}^{-a_{ji}} s_j(u-\tfrac{1}{2}d_i a_{ij} - r d_j)}{s_i(u) s_i(u-d_i)} = u^{\langle \mu,\alpha_i\rangle}\frac{u^{m_i} (u-d_i)^{m_i} }{p_i(u) \prod_{j \ne i}\prod_{r=1}^{-a_{ji}} (u-\tfrac{1}{2}d_i a_{ij} - rd_j )^{m_j}}  
\end{equation*}
Using (\ref{eq: def of A gens without R}), we see that the series $s_i(u) \gkloA_i(u)$ satisfy the defining equation (\ref{eq: def of A gens}) of $A_i(u)$.  By the uniqueness of $A_i(u)$, we conclude $A_i(u) = s_i(u) \gkloA_i(u)$.
\end{proof}

Recall that in the previous section we described elements $t_{\beta,w}^{(r)}$ in the algebra $Y = Y_0$.  
For each $i\in I$, fix a highest weight vector $v_i \in V(\varpi_i, 0)$.  Denote its dual $v_i^\ast \in V(\varpi_i,0)^\ast$. Note that the dual of $f_i v_i$ is $- e_i v_i^\ast$.

\begin{Proposition}[\mbox{\cite[Prop.~2.29]{IR}}] 
\label{prop: RTT vs GKLO}
For any $i\in I$, in the algebra $Y_0$ we have
$$
t_{v_i^\ast, v_i}(u) = \gkloA_i(u), \qquad t_{-e_i v_i^\ast, v_i}(u) = \gkloB_i(u), \qquad t_{v_i^\ast, f_i v_i}(u) = \gkloC_i(u)
$$
\end{Proposition}
\begin{Rem}
Note that there is a small error in the statement of \cite[Prop.~2.29]{IR}, with $\gkloB_i(u)$ and $\gkloC_i(u)$ being transposed.  We thank Aleksei Ilin and Lenya Rybnikov for discussion on this point. 
\end{Rem}

For any $\mu$, the algebra $Y_\mu$ admits a grading $Y_\mu = \bigoplus_{\gamma \in Q} Y_\mu(\gamma)$ by the root lattice $Q = \bigoplus_{i \in I} \mathbb{Z} \alpha_i$.  It is defined by assigning degrees 
$$
\deg E_i^{(r)} = \alpha_i, \qquad \deg H_i^{(r)} = 0, \qquad \deg F_i^{(r)} = - \alpha_i
$$
The subalgebras $Y_\mu^>, Y_\mu^\geq$, etc.~are all homogeneous for this grading.  It is also easy to see that this grading is preserved by the shift morphisms $\iota_{\mu, \mu_1, \mu_2}$ and that comultiplications preserve total degree.  One can also see that there is an induced grading on any quotient $Y_\mu\twoheadrightarrow Y_\mu^\lambda(\bR)$.
\begin{Lemma}
\label{lemma: RTT weights}
If $\beta \in V^\ast$ and $w\in V$ are both weight vectors, then
$$t_{\beta, w}^{(r)} \in Y_0( \wt \beta + \wt \gamma)$$
\end{Lemma}
\begin{proof}
This follows from the formulas for commutators just before \cite[Cor.~2.30]{IR}.
\end{proof}

\begin{Proposition}
\label{lem: estimate coproduct}
Under the comultiplication map $\Delta: Y_\mu \rightarrow Y_{\mu_1} \otimes Y_{\mu_2}$, for any $i \in I$ we have
$$
\Delta\big( \gkloA_i(u) \big) = \gkloA_i(u)\otimes \gkloA_i(u) + \gkloC_i(u) \otimes \gkloB_i(u) + \bigoplus_\gamma  Y_{\mu_1}^\leq(-\gamma) \otimes Y_{\mu_2}^\geq (\gamma) [[ u^{-1}]],
$$
where the sum ranges over $\gamma \in Q$ such that $\gamma > \alpha_i$ and the weight space $V(\varpi_i)_{\varpi_i-\gamma} \neq 0$.
\end{Proposition}
\begin{proof}
We first study the case $\mu = \mu_1 = \mu_2 = 0$.  Apply Lemma \ref{lem: RTT coproduct}, taking any basis of weight vectors for $V(\varpi_i, 0)$ which contains $\{v_i, f_i v_i\}$.   By Proposition \ref{prop: RTT vs GKLO} the two summands corresponding to $v \in \{v_i, f_i v_i\}$ are precisely $\gkloA_i(u) \otimes \gkloA_i(u) + \gkloC_i(u)\otimes \gkloB_i(u)$. It remains to show that all other summands have degrees for the root lattice grading as claimed. As a $\g$--representation $V(\varpi_i, 0) \cong V(\varpi_i)\oplus \bigoplus_\nu V(\nu)$ where $\nu < \varpi_i$ (with appropriate multiplicities), see \cite[Section 12.1 C]{CP}.   Thus the weights of $V(\varpi_i,0)$  are the same as the weights of $V(\varpi_i)$.  Finally, if $v \in V(\varpi_i, 0)$ has weight $\varpi_i - \gamma$, then $ t_{v_i^\ast, v}^{(r)} \in Y(-\gamma)$ and $t_{v^\ast, v_i}^{(r)} \in Y(\gamma)$ by Lemma \ref{lemma: RTT weights}.  This proves the claim.

Next, we assume that $\mu, \mu_1, \mu_2$ are all antidominant.  By Theorem \ref{coproduct}, there is a commutative diagram
\[
\xymatrix{
Y_0  \ar[d]_{\iota_{0,\mu_1, \mu_2}} \ar[rrr]^{\Delta_{0,0}} &&& Y_{0}\otimes Y_{0} \ar[d]^{(\iota_{0,\mu_1,0})\otimes (\iota_{0,0,\mu_2})}\\
Y_{\mu} \ar[rrr]_{\Delta_{\mu_1, \mu_2}} &&& Y_{\mu_1}\otimes Y_{\mu_2}
}
\]
We know that the proposition holds for the top row.  Using Lemma \ref{lemma: GKLO ABC and shifts} and the observation that 
$$
(\iota_{0,\mu_1,0}\otimes \iota_{0,0,\mu_2})\big(Y_0^\leq (-\gamma) \otimes Y_0^\geq (\gamma)\big) \subseteq Y_{\mu_1}^\leq (-\gamma) \otimes Y_{\mu_2}^\geq (\gamma),
$$
 we see that proposition also holds for the bottom row. 

Finally, for the case of arbitrary $\mu, \mu_1, \mu_2$, we again apply Theorem \ref{coproduct}, but now with $\eta_1, \eta_2$ very antidominant.  Then we already know that the proposition holds for the bottom row. So it also holds for the top row, by the injectivity of the vertical arrows and Lemma \ref{lem:shiftrespectsborel}. This completes the proof.
\end{proof}

%%%%
\subsection{Explicit formulas for Cartan generators}
\label{section: explicit formulas for Cartan generators}

In this section, we will use Proposition \ref{lem: estimate coproduct} to obtain the following explicit formulas for the composed map
\[
\xymatrix{
Y_\mu \ar[r]^<<<<\Delta &  Y_{-\alpha_i^\vee} \otimes Y_{\mu+\alpha_i^\vee}  \ar@{->>}[r] & Y_{-\alpha_i^\vee}^0 \otimes Y_{\mu+\alpha_i^\vee}
}
\]
Fix a dominant coweight $\lambda$ such that $\lambda \geq \mu$ and $\lambda \geq \mu+\alpha_i^\vee$, as well as a set of parameters $\bR$ of coweight $\lambda$.  Then we may define elements $A_i^{(r)} \in Y_\mu$ depending on the data $\lambda, \bR$, as well as elements $A_i^{(r)} \in Y_{-\alpha_i^\vee}$ depending on the data $0, \emptyset$, and finally elements $A_i^{(r)} \in Y_{\mu+\alpha_i^\vee}$ depending on $\lambda, \bR$.

\begin{Theorem}
\label{thm: explicit comult}
Under the above map $Y_{\mu} \rightarrow Y^0_{-\alpha_i^\vee} \otimes Y_{\mu+\alpha_i^\vee}$, we have
\begin{align*}
A_i(u)  & \longmapsto A_i(u) \otimes A_i(u) + d_i [A_i(u), F_i^{(1)}] \otimes [E_i^{(1)}, A_i(u)], \\
 A_j(u) & \longmapsto 1 \otimes A_j(u)
\end{align*}
for any $j\neq i$.  
\end{Theorem}

This result easily implies Theorem \ref{th: weak form of conjecture}:
\begin{proof}[Proof of Theorem \ref{th: weak form of conjecture}]
In Theorem \ref{thm: explicit comult}, the right-hand sides are polynomials in $u^{-1}$ of degree $m_j$, for any $j\in I$ (including $j=i$). Thus $A_j^{(r)}$ is sent to zero for $r > m_j$.
\end{proof}

In order to prove Theorem \ref{thm: explicit comult}, we first establish two lemmas. For the first it is helpful to work more generally, so we fix some notation. Let $\lambda = \lambda_1 + \lambda_2$ be dominant coweights, and $\mu = \mu_1 + \mu_2$ with $\mu_i \leq \lambda_i$.  Choose sets of parameters $\bR, \bR_1, \bR_2$ corresponding to $\lambda, \lambda_1, \lambda_2$, resp.  We will assume that that $\bR = \bR_1 \sqcup \bR_2$.  That is, writing $\bR = (R_i)_{i\in I}$ and $\bR_a = (R_{a, i})_{i\in I}$ for $a=1,2$, we assume that $R_i = R_{1,i} \sqcup R_{2,i}$ is a multiset union for each $i\in I$.  

Given these data, we define corresponding elements $A_i^{(r)} \in Y_\mu$ with respect to $\lambda, \bR$ (resp.~in $Y_{\mu_1}$ for $\lambda_1, \bR_1$, and in $ Y_{\mu_2}$ for $\lambda_2, \bR_2$) as in section \ref{section: tsy}.  We wish to understand the relation between these elements $A_i^{(r)}$ under the comultiplication map $\Delta: Y_\mu \rightarrow Y_{\mu_1} \otimes Y_{\mu_2}$:

\begin{Lemma}
\label{lemma: key estimate of comult}
With notation as above, for any $i \in I$, 
$$
\Delta\big( A_i(u) \big) = A_{i}(u) \otimes A_{i}(u) + d_i [A_{i}(u), F_i^{(1)}] \otimes [E_i^{(1)}, A_i(u)] + \bigoplus_\gamma  Y_{\mu_1}^\leq(-\gamma) \otimes Y_{\mu_2}^\geq (\gamma) [[ u^{-1}]],
$$
where the sum ranges over $\gamma \in Q$ such that $\gamma > \alpha_i$, and $V(\varpi_i)_{\varpi_i-\gamma} \neq 0$.
\end{Lemma}
\begin{proof}
By Lemma \ref{lemma: various GKLO generators }, we have equalities $A_i(u) = s_i(u) \gkloA_i(u)$ in $Y_\mu$, $A_{i}(u) = s_{1,i}(u) \gkloA_i(u)$ in $Y_{\mu_1}$, and $A_{i}(u) = s_{2,i}(u) \gkloA_i(u)$ in $Y_{\mu_2}$, for some series $s_i(u), s_{1,i}(u), s_{2,i}(u) \in 1 + u^{-1} \C[[u^{-1}]]$.  Moreover, these series satisfy
\begin{equation}
\label{eq: multiplier series}
s_i(u) = s_{1,i}(u) s_{2,i}(u)
\end{equation}
for all $i\in I$. Indeed, the product of the equations defining the $s_{1,i}(u)$ and $s_{2,i}(u)$ is the equation defining the $s_i(u)$, because $\lambda = \lambda_1 + \lambda_2$, $\bR = \bR_1 \sqcup \bR_2$ and $\mu = \mu_1 + \mu_2$. So (\ref{eq: multiplier series}) holds by the uniqueness of the $s_i(u)$.  Next,  we apply Lemma \ref{lemma: gklo relations} to replace $\gkloC_i(u)\otimes\gkloB_i(u)$ by $d_i[\gkloA_i(u), F_i^{(1)}] \otimes [E_i^{(1)}, \gkloA_i(u)]$ in Proposition \ref{lem: estimate coproduct}.  Multiplying both sides by $s_i(u)$ and using (\ref{eq: multiplier series}), we deduce the claim.
\end{proof}

Recall the algebra $Y_{-\alpha_i^\vee}^0$, as described explicitly in Lemma \ref{lemma: simplest coulomb branch of quiver gauge theory}. 
\begin{Lemma}
\label{Lemma: big kernel for Yalpha0}
Let $\gamma \in Q$, with $\gamma \notin \mathbb{Z} \alpha_i$.  Then the graded component $Y_{-\alpha_i^\vee}(- \gamma) $ is in the kernel of the quotient map $Y_{-\alpha_i^\vee} \twoheadrightarrow Y_{-\alpha_i^\vee}^0$.
\end{Lemma}
\begin{proof}
By definition, for $j \neq i$ the generators $A_j^{(r)}, E_j^{(r)}, F_j^{(r)}$ all map to zero in $Y_{-\alpha_i^\vee}^0$.
\end{proof}

\begin{proof}[Proof of Theorem \ref{thm: explicit comult}]
Let $j \in I$, and let $\gamma \in Q$ with $\gamma > \alpha_j$ and $V(\varpi_j)_{\varpi_j - \gamma} \neq 0$.  Then $\gamma$ cannot be an integer multiple of $\alpha_i$. Indeed, if $j\neq i$ then we have already assumed $\gamma > \alpha_j$.  Meanwhile, if $j = i$ then we know $\gamma > \alpha_i$, so we would have $\gamma \geq 2 \alpha_i$.  But then $V(\varpi_i)_{\varpi_i - \gamma} = 0$.  

It follows that the image of $Y_{-\alpha_i^\vee}(-\gamma)$ in $Y_{-\alpha_i^\vee}^0$ is zero by Lemma \ref{Lemma: big kernel for Yalpha0}.  Together with Lemma \ref{lemma: key estimate of comult}, this immediately establishes the claim for the case $j=i$.  When $j \neq i$, the image of $A_j^{(r)} $ in $Y_{-\alpha_i^\vee}^0$ is zero for $r\geq 1$, and thus $A_j(u) \mapsto 1$, establishing this case as well. 
\end{proof}

%%%%%%%%%%%%%%%%%%%%%%%%%%%%%

\appendix

%%%
\section{Defining ideals of generalized slices}
\label{section: appendix}

This appendix can be considered as an extension of the results of \cite{KMW}, since Theorem \ref{appendix thm} gives an explicit description of the defining ideals of $\oW^\lambda_\mu \subset \cW_\mu$ and of $Y_\mu \twoheadrightarrow Y_\mu^\lambda(\bR)$.  These descriptions are conditional on a ``reducedness'' conjecture for a modular description of the spherical Schubert variety $\overline{\Gr^\lambda}$.  

Along the way, we prove several results about the Poisson structure on $\cW_\mu$ for $\mu$ antidominant, which may be of independent interest.

\subsection{Notation}
\label{appendix: notation}
For a weight $\eta$, denote $\eta^\ast = - w_0 \eta$ where $w_0 \in W$ is the longest element of the Weyl group.   

Locally to this appendix, we fix $v_i \in V(\varpi_i^\ast)_{-\varpi_i}$ a lowest weight vector and $v_i^\ast \in V(\varpi_i^\ast)^\ast_{\varpi_i}$ its dual.  For any $\beta \in V(\varpi_i^\ast)^\ast$ and $v\in V(\varpi_i^\ast)$, we have a matrix coefficient $\Delta_{\beta, v} \in \C[G]$.  The group $G((t^{-1}))$ acts on $V(\varpi_i^\ast) ((t^{-1}))$, and for $r \in \mathbb{Z}$ we define $\Delta_{\beta, v}^{(r)} \in \C( G((t^{-1})) )$ by 
$$
\Delta_{\beta,v}(g) = \langle \beta, g v\rangle = \sum_{r \in \mathbb{Z}} \Delta_{\beta, v}^{(r)}(g) t^{-r}
$$
It is convenient to encode these functions using generating series 
$$
\Delta_{\beta, v}(u) = \sum_{r \in\mathbb{Z}} \Delta_{\beta, v}^{(r)} u^{-r} \in \C[ G((t^{-1})) ] [[ u, u^{-1}]]
$$
Note that for any $g = g(t) \in G((t^{-1}))$, this generating series evaluates to the Laurent series $\langle\beta, g(u) v \rangle \ \in \ \C((u^{-1}))$.

By restricting along the inclusion $\cW_\mu \subset G((t^{-1}))$, we will also view $\Delta_{\beta, v}^{(r)} \in \C[\cW_\mu]$.  With this notation, we can restate the isomorphism $\gr Y_\mu \cong \C[\cW_\mu]$ from Theorem \ref{grYmu}:
\begin{align*}
A_i(u)  = u^{\langle \mu, \varpi_i\rangle} \Delta_{v_i^\ast, v_i}(u), \quad E_i(u) = d_i^{-1/2} \frac{ \Delta_{-f_i v_i^\ast, v_i}(u)}{\Delta_{v_i^\ast, v_i}(u)}, \quad F_i(u)  = d_i^{-1/2} \frac{ \Delta_{v_i^\ast, e_i v_i}(u)}{\Delta_{v_i^\ast, v_i}(u)}
\end{align*}
Indeed, one can easily check that for $\Delta_{-f_i v_i^\ast, v_i} = \psi_i$ as functions on $U$, and so on.

\begin{Rem}
Our conventions here differ from those used section \ref{section: RTT presentation}, see e.g.~Proposition \ref{prop: RTT vs GKLO}.  These reflect opposite choices of order for Gauss decompositions: the ordering $U T U_-$ used here  and in section \ref{se:WmuDef}, which follows \cite[Section 2(xi)]{BFN}, versus the ordering $U_- T U$ reflected in section \ref{section: RTT presentation}, which follows the conventions of \cite{IR}.  Either convention could be changed.  We trust that this conflict will cause no confusion to the reader.
\end{Rem}

%%%
\subsection{Defining ideals and main result}
\label{section: defining ideals, reducedness conjecture}
A modular description of any spherical Schubert variety $\overline{\Gr^\lambda} \subset \Gr_G$ was proposed by Finkelberg-Mirkovi\'c \cite{FM}: a moduli space $\overline{\mathcal{Y}^\lambda} \subset \Gr_G$ of principal $G$-bundles on $\mathbb P^1$, with a trivialization on $ \mathbb A^1 $, satisfying a pole condition which depends on $\lambda$. See \cite{KMW} for more details.   This description is set-theoretically correct, but it is not clear that $\overline{\mathcal{Y}^\lambda}$ is reduced.  This leads to the following reducedness conjecture:

\begin{Conjecture}[\mbox{\cite[Conjecture 1.1]{KMW}}]
\label{conjecture: reducedness}
For any dominant coweight $\lambda$, the scheme $\overline{\mathcal{Y}^\lambda}$ is reduced.  In particular, $\overline{\mathcal{Y}^\lambda} = \overline{\Gr^\lambda}.$
\end{Conjecture}

\begin{Theorem}[\mbox{\cite{KMWY}}]
Conjecture \ref{conjecture: reducedness} holds in type A.
\end{Theorem}

Recall that $\oW^\lambda_\mu$ is defined to be the intersection $\cW_\mu \cap \overline{G[t] t^\lambda G[t]}$.  The locus $\overline{G[t] t^\lambda G[t]}$ is the preimage of $\overline{\Gr^\lambda}$ under projection $G((t^{-1})) \rightarrow G((t^{-1}))/G[t]$. 

If we apply the Finkelberg-Mirkovi\'c pole condition, then we immediately obtain the following result.

\begin{Theorem}
	\begin{enumerate}
		\item $\oW^\lambda_\mu \subset \cW_\mu $ is set-theoretically defined by the vanishing of the elements $\Delta_{\beta, v}^{(r)} \in \C[\cW_\mu]$ over all $i \in I$, vectors $\beta \in V(\varpi_i^\ast)^\ast$, $v\in V(\varpi_i^\ast)$ and $r > \langle \lambda, \varpi_i \rangle$. 
		\item If $ \overline{\mathcal{Y}^\lambda} $ is reduced, then the above description of $ \oW^\lambda_\mu $ holds scheme-theoretically.
	\end{enumerate}
\end{Theorem}

The following is the main result of this appendix:  
\begin{Theorem}
\label{appendix thm}
Assume that $\overline{\mathcal{Y}^\lambda}$ is reduced. Then for any coweight $\mu \leq \lambda$, 
\begin{enumerate}
\item[(a)] $\oW^\lambda_\mu \subset \cW_\mu$ is scheme-theoretically defined by the ideal Poisson generated by the elements $\Delta_{v_i^\ast, v_i}^{(r)}$ over all $i\in I$ and $r > \langle \lambda, \varpi_i \rangle$.
\item[(b)] For any $\bR$, there is an isomorphism $\conjectural_\mu^\lambda(\bR) \cong Y_\mu^\lambda(\bR)$.
\end{enumerate}
In particular, these claims hold in type A.
\end{Theorem}
The proof of Theorem \ref{appendix thm} will be given in section \ref{section: proof of appendix main theorem}.  This theorem generalizes \cite[Theorem 1.6]{KMW}, which covers the case when $\mu$ is dominant.   Our proof will follow the same strategy, which requires some explicit control of Poisson brackets.

%%%
\subsection{Poisson structures} Recall from Theorem \ref{grYmu} that the isomorphism  $\gr Y_\mu \cong \C[\cW_\mu]$ endows $\cW_\mu$ with a Poisson structure.   Our goal in this section is to give another description of this Poisson structure, in the case when $\mu$ is antidominant.  

We begin by recording the following general result, due to Semenov-Tian-Shansky \cite[Theorem 7]{STS}, see also \cite[Theorem 1.9]{Yakimov}.  To fix notation, let $D$ be an algebraic group over $\C$, with closed subgroups $P, Q$.  Assume that their Lie algebras $(\mathfrak{d}, \mathfrak{p}, \mathfrak{q})$ form a Manin triple.  In this case we will also call $(D, P, Q)$ a Manin triple.  This endows $D$ with a natural Poisson structure via the Skylanin bracket \cite[Section 2.2A]{CP}, which is the double of $P$. 
\begin{Proposition}
\label{prop: STS symplectic leaves}
The symplectic leaves of $D$ are the connected components of double coset intersections $P x P \cap Q y Q$, where $x, y \in D$.  
\end{Proposition}

As discussed in section \ref{section: Poisson structure of Wmu and quantizations}, there is a Manin triple of loop groups $(G((t^{-1})), G_1[[t^{-1}]], G[t])$. Since the above proposition is essentially a reinterpretion of the dressing action in the case of a Poisson double $D$, it still applies in this infinite dimensional setting; we may  at least conclude that connected components of any \emph{finite dimensional} intersection 
\begin{equation}
\label{eq: leaves of loop group}
G_1[[t^{-1}]] x G_1[[t^{-1}]] \cap G[t] y G[t]
\end{equation}
are symplectic leaves of $G((t^{-1}))$.

\begin{Rem}
Fix dual bases $\{J_a\}$ and $\{J^a\}$ for $\g$ with respect to the bilinear form $(\cdot, \cdot)$ from section \ref{section:notation}.  Then as in \cite[Proposition 2.13]{KWWY}, the Poisson bracket is given by 
\begin{equation}
\label{eq: loop group Poisson structure}
(u-v)\left\{ \Delta_{\beta_1, v_1}(u), \Delta_{\beta_2, v_2}(v) \right\} = \sum_a \Big( \Delta_{\beta_1, J_a v_1}(u) \Delta_{\beta_2, J^a v_2}(v) - \Delta_{J_a \beta_1, v_1}(u) \Delta_{J^a \beta_2, v_2}(v) \Big)
\end{equation}
This should be interpreted as an equality of formal series in $\C[G((t^{-1}))] [[u^{\pm 1}, v^{\pm 1}]]$.   One can show that this formula is compatible with the ind-scheme structure on  $G((t^{-1}))$.
\end{Rem}

\begin{Theorem}
\label{theorem: antidominant poisson structure}
Let $\mu$ be an antidominant coweight.  Then with respect to the above Poisson structure on $G((t^{-1}))$:
\begin{enumerate}
\item[(a)] $\cW_\mu \subset G((t^{-1}))$ is a closed Poisson subscheme.

\item[(b)] For any dominant $\lambda \geq \mu$, $\oW^\lambda_\mu \subset G((t^{-1}))$ is a closed Poisson subscheme.
\end{enumerate}
\end{Theorem}
\begin{proof}
First we observe that $\cW_\mu \subset G((t^{-1}))$ is always a closed subscheme, even if $\mu$ is not antidominant, as follows easily from \cite[Lemma 3.3]{MW}.  Since every $\oW^\lambda_\mu \subset \cW_\mu$ is also closed, this proves the `closed' part of both statements.

Since $\mu$ is antidominant, we claim that
\begin{equation*}
\cW_\mu = G_1[[t^{-1}]] t^\mu G_1[[t^{-1}]]
\end{equation*}
Indeed, this follows  from the Gauss decomposition $G_1[[t^{-1}]] = U_1[[t^{-1}]] T_1[[t^{-1}]] U_{-,1}[[t^{-1}]]$, together with the fact that $t^\mu U_1[[t^{-1}]] t^{-\mu} \subset U_1[[t^{-1}]]$ and $t^{-\mu} U_{-,1}[[t^{-1}]] t^\mu \subset U_{-,1}[[t^{-1}]]$ since $\mu$ is antidominant.  

Therefore for any dominant $\nu$ with $\nu \geq \mu$,
$$
\cW^\nu_\mu = \cW_\mu \cap G[t]t^\nu G[t] = G_1[[t^{-1}]] t^\mu G_1[[t^{-1}]] \cap G[t] t^\nu G[t]
$$
is a symplectic leaf of $G((t^{-1}))$ by equation (\ref{eq: leaves of loop group}).  As in (\ref{eq: symplectic leaves}), we see that $\oW^\lambda_\mu$ is a finite union of symplectic leaves, so is a Poisson subvariety.  This proves part (b).  Since the collection $\{\bigcup_\lambda \oW_\mu^\lambda\}$ fills-out $\cW_\mu$ by Proposition \ref{denseunion}, it follows that $\cW_\mu$ is Poisson, proving part (a).

\end{proof}

\begin{Rem}
\label{rem: wary of infinite dimensional}
One can also use equation (\ref{eq: loop group Poisson structure}) to show explicitly that the defining ideals of $\cW_\mu$ and $\oW^\lambda_\mu$ are Poisson, which gives another proof of the theorem.
\end{Rem}

For $\mu$ antidominant, the above theorem endows $\cW_\mu$ and its subvarieties $\oW^\lambda_\mu$ with Poisson structures.  Recall that back in section \ref{section: Poisson structure of Wmu and quantizations}, we defined another Poisson structure on $ \cW_\mu $ using its quantization by the shifted Yangian.  Our goal for the remainder of this section is to show that these two Poisson structures on $\cW_\mu$ agree.  Note that for $\mu = 0$ this is clear, since $\cW_0 = G_1[[t^{-1}]]$ is a Poisson subgroup of $G((t^{-1}))$ and is quantized by 
the Yangian $Y_0$ by Theorem \ref{theorem: yangian quantizes loop group}.

We begin by establishing an analogue of Proposition \ref{def: shift geometric} for the new Poisson structure.

\begin{Lemma}
\label{lemma: shift is poisson v2}
For any antidominant $\mu$, the shift map $\iota_{0,0,\mu}: \cW_\mu \rightarrow \cW_0$ from Definition \ref{def: shift geometric} is Poisson.
\end{Lemma}
\begin{proof}
The actions of $G_1[[t^{-1}]]$ on $\cW_\mu$ by left or right multiplication are Poisson, since $\cW_\mu$ is a Poisson subscheme of $G((t^{-1}))$ and $G_1[[t^{-1}]]$  its Poisson subgroup. Now, consider the following subgroup $H$ of $G_1[[t^{-1}]]$:
$$
H = G_1[[t^{-1}]] \cap t^{-\mu} G[t] t^\mu = U_{-,1}[[t^{-1}]] \cap t^{-\mu}U_{-}[t] t^\mu$$
With respect to the action of $H$ by right multiplication there is an isomorphism $\cW_\mu / H \cong \cW_0$ defined by $g H \mapsto \pi(g t^{-\mu})$, which identifies $\iota_{0,0,\mu}$ with the quotient map $\cW_\mu \rightarrow \cW_\mu / H$.  Indeed, since $U_{-,1}[[t^{-1}]]$ is pro-unipotent there is a decomposition
\begin{align*}
U_{-,1}[[t^{-1}]] &= \big( U_{-,1}[[t^{-1}]] \cap t^{-\mu} U_{-,1}[[t^{-1}]] t^{\mu}\big) \times \big( U_{-,1}[[t^{-1}]] \cap t^{-\mu} U_{-}[t] t^{\mu}\big) \\
& = t^{-\mu} U_{-,1}[[t^{-1}]] t^\mu \times H
\end{align*}
The second equality follows from the definition of $H$, together with the fact that $-\mu$ is dominant so $t^{-\mu} U_{-,1}[[t^{-1}]] t^\mu \subseteq U_{-,1}[[t^{-1}]]$. Using this decomposition, we see that the quotient $\cW_\mu \rightarrow \cW_\mu/H$ can be identified with the projection onto the first factor in
$$
\cW_\mu = \big( U_1[[t^{-1}]] T_1[[t^{-1}]] t^\mu ( t^{-\mu} U_{-,1}[[t^{-1}]] t^\mu) \big) \times H
$$
The map $gH \mapsto \pi (g t^{-\mu})$ is then identified with the isomorphism given by right multiplication by $t^{-\mu}$:
$$
U_1[[t^{-1}]] T_1[[t^{-1}]] t^\mu (t^{-\mu}U_{-,1}[[t^{-1}]] t^\mu) \stackrel{\sim}{\longrightarrow} U_1[[t^{-1}]] T_1[[t^{-1}]] U_{-,1}[[t^{-1}]]
$$

One can also check using the Lie cobracket on $t^{-1}\g[[t^{-1}]]$ that $H$ satisfies the conditions of \cite[Theorem 6]{STS} (cf.~the proof of \cite[Prop.~5.14]{FKPRW})
\begin{equation}
\label{cobracket}
\delta( \operatorname{Lie} H) \subset  (t^{-1}\g[[t^{-1}]]) \otimes \operatorname{Lie} H+  \operatorname{Lie} H\otimes (t^{-1}\g[[t^{-1}]])
\end{equation}
This implies that there exists a unique Poisson structure on $\cW_\mu / H$ such that the quotient map is Poisson. Moreover the residual left action of $G_1[[t^{-1}]]$ on $\cW_\mu / H$ is  Poisson. Via the isomorphism above, we conclude that there is a unique Poisson structure on $\cW_0 $ making $\iota_{0,0,\mu}$ Poisson, and that the left $G_1[[t^{-1}]]$ action is Poisson for this structure.  

To complete the proof, we claim that this Poisson structure on $\cW_0 = G_1[[t^{-1}]]$ is simply the standard one.  Since $\cW_0$ is a Poisson homogeneous space for $G_1[[t^{-1}]]$, by \cite[Theorem 1]{Drinfeld PHS} it suffices to show that the corresponding Poisson bivector vanishes at $t^0 \in \cW_0$.  Since $\iota_{0,0,\mu}(t^\mu) = t^0$, it further suffices to check that the image of the Poisson bivector at $t^\mu \in \cW_\mu$ vanishes under the pushforward map $(\iota_{0,0,\mu})_\ast:\bigwedge^2T_{t^\mu} \cW_\mu \rightarrow \bigwedge^2 T_{t^0} \cW_0$, see \cite[Equation (7.2)]{Vaisman}.  Based on the above analysis of the map $\iota_{0,0,\mu}$, this will follow if we can show that the Poisson bivector at $t^\mu \in \cW_\mu$ lies in 
$$
 T_{t^\mu} \cW_\mu \otimes T_{t^\mu} H +  T_{t^\mu} H \otimes T_{t^\mu} \cW_\mu
$$
This follows by a similar calculation to (\ref{cobracket}), cf.~\cite[Section 2.2A]{CP}.  This completes the proof.

\end{proof}

Armed with this result, we can easily complete our goal. 
\begin{Theorem}
If $\mu$ is antidominant, then $Y_\mu$ quantizes the Poisson structure on $\cW_\mu$ from Theorem \ref{theorem: antidominant poisson structure}, and  $Y^\lambda_\mu(\bR)$ quantizes the corresponding Poisson structure on $\oW^\lambda_\mu$, for any set of parameters $\bR$ of coweight $\lambda$.
\end{Theorem}

\begin{proof}
We have two a priori different Poisson structures to compare.  Since the $\oW_\mu^\lambda$ are Poisson subvarieties with respect to both structures, and because the collection $\{\oW_\mu^\lambda\}$ fills-out $\cW_\mu$ by Proposition \ref{denseunion}, it suffices for us to show that the two Poisson structures agree after restriction to every $\oW_\mu^\lambda$. 

Consider the restriction $\iota_{0,0,\mu}: \oW^\lambda_\mu \rightarrow \oW^{\lambda-\mu}_0$.  It is a birational map, which is Poisson for both structures by Proposition \ref{shiftpoi} and Lemma \ref{lemma: shift is poisson v2}, respectively.   Since the two Poisson structures agree on $\oW^{\lambda-\mu}_0$ by Theorem \ref{theorem: yangian quantizes loop group}, Lemma \ref{birpoi} implies that the two Poisson structures agree on $\oW^\lambda_\mu$ as well, proving the claim.

\end{proof}

%%%
\subsection{Proof of Theorem \ref{appendix thm}}
\label{section: proof of appendix main theorem}

Throughout this section we let $\mu$ be arbitrary coweight, and will assume that $\nu$ is an antidominant coweight such that $\mu-\nu$ is dominant.  (In particular if $\mu$ is already antidominant, we may take $\nu = \mu$.) We also let $\lambda \geq \mu$ be a dominant coweight.  We will apply the same arguments from \cite[Section 9]{KMW}, but utilizing the shift map
\begin{equation}
\label{eq: p embedding}
p  = \iota_{\mu,0, \nu-\mu}: \cW_{\nu} \rightarrow \cW_\mu
\end{equation}
The advantage of choosing $\nu$ antidominant is that, by the results of the previous section,  the Poisson bracket on $\cW_\nu$ is encoded explicitly by the formula (\ref{eq: loop group Poisson structure}).  The proofs in \cite[Section 9]{KMW} in the $\nu = 0$ case are based explicitly around this formula, and generalize to the present setting immediately.

Let $J^\lambda_\mu \subset \C[\cW_\mu]$ denote the Poisson ideal generated by the elements $\Delta_{v_i^\ast, v_i}^{(r)}$ with $i\in I$ and $r > \langle \lambda, \varpi_i \rangle$.  Considering $\C[\cW_\mu] $ as a subalgebra of $\C[\cW_\nu]$ via the embedding (\ref{eq: p embedding}), let $I^\lambda_\mu = J^\lambda_\mu \cdot \C[\cW_\nu]$ denote the ideal of $\C[\cW_\nu]$ generated by $J^\lambda_\mu$.  

\begin{Proposition}
$J^\lambda_\mu$ is the defining ideal of $\oW_\mu^\lambda$ as a subvariety of $\cW_\mu$ if and only if $I^\lambda_\mu$ is the defining ideal of $p^{-1}(\oW_\mu^\lambda)$ as a subvariety of $\cW_\nu$.
\end{Proposition}
\begin{proof}
As in the proof of Lemma \ref{lemma: shift is poisson v2}, we can identify $\cW_\nu \cong \cW_\mu \times H$ where $H = U_{-,1}[[t^{-1}]] \cap t^{\mu-\nu} U_- [t] t^{\nu-\mu}$ is a finite-dimensional unipotent group; in particular $H$ is isomorphic to a finite-dimensional affine space.  The map $p$ is identified with projection onto the first factor, so $p$ is a trivial affine fibration. With this in mind, the proof of \cite[Proposition 8.6]{KMW} applies.
\end{proof}

\begin{Proposition}
\label{prop: intermediate KMW}
$I^\lambda_\mu \subset \C[\cW_\nu]$ is generated as an ordinary ideal by 
$$
\Delta_{\beta, \gamma}^{(r)} \quad \text{for } r > \langle \mu, \varpi_i \rangle + \langle \mu - \nu, \operatorname{wt}(\gamma)\rangle
$$
over all $i \in I$, and where $\beta \in V(\varpi_i^\ast)^\ast$ and $\gamma\in V(\varpi_i^\ast)$ range over weight bases, respectively.
\end{Proposition}
\begin{proof}
\cite[Lemma 9.3]{KMW} holds in the current setting, and the proof of \cite[Proposition 9.4]{KMW} generalizes to prove the claim.
\end{proof}

We may use the scheme $\overline{\mathcal{Y}^\lambda}$ to define a possibly non-reduced version of $\overline{\cW}_\mu^\lambda$: the scheme-theoretic intersection of $\cW_\mu$ with the inverse image of $\overline{\mathcal{Y}^\lambda}$ in $G((t^{-1}))$. Let us denote the resulting closed subscheme $\overline{\mathcal{Y}}_\mu^\lambda \subset \cW_\mu$. Then the corresponding induced reduced scheme is $(\overline{\mathcal{Y}}_\mu^\lambda)_{red} = \overline{\cW}_\mu^\lambda$. Moreover, if $\overline{\mathcal{Y}^\lambda}$ is reduced then $\overline{\mathcal{Y}}_\mu^\lambda = \overline{\cW}_\mu^\lambda$.
\begin{Corollary}
\label{cor: radical of poisson ideal}
$J_\mu^\lambda$ is the defining ideal of $\overline{\mathcal{Y}}_\mu^\lambda \subset \cW_\mu$.  In particular, its radical $\sqrt{J_\mu^\lambda}$ is the defining ideal of $\overline{\cW}_\mu^\lambda \subset \cW_\mu$.
\end{Corollary}
\begin{proof}
Follows from Proposition \ref{prop: intermediate KMW}, via the same proof as \cite[Theorem 8.7]{KMW}.
\end{proof}

\begin{proof}[Proof of Theorem \ref{appendix thm}]
By assumption $\overline{\mathcal{Y}^\lambda} = \overline{\Gr^\lambda}$.  Then $\overline{\mathcal{Y}}_\mu^\lambda = \overline{\cW}_\mu^\lambda$ so $J_\mu^\lambda = \sqrt{J_\mu^\lambda}$ by the previous corollary. This proves (a).

For part (b), recall from the proof of Proposition \ref{prop: conjectural yangian quantizes scheme} that the kernel of $\C[\cW_\mu] \twoheadrightarrow \gr \conjectural_\mu^\lambda(\bR)$ is sandwiched between $J_\mu^\lambda$ and the defining ideal of $\overline{\cW}_\mu^\lambda$.  Since $\overline{\mathcal{Y}}_\mu^\lambda = \overline{\cW}_\mu^\lambda$, these ideals are equal by the previous corollary, and so are also equal to the above kernel.  Therefore $\conjectural_\mu^\lambda(\bR)$ quantizes $\overline{\cW}_\mu^\lambda$.
\end{proof}

\end{document}